\documentclass[11pt,letterpaper]{article}
\usepackage{amsfonts, amsmath, amssymb, amscd, amsthm, color, graphicx, mathabx, mathrsfs, wasysym, setspace, mdwlist, calc, float, mathtools, tikz-cd,enumitem}
\usepackage{setspace}

\usepackage{enumitem}
\usepackage{tocloft}
\usepackage{xcolor}

\usepackage{hyperref}
\hypersetup{linktocpage}

\hypersetup{colorlinks,
    linkcolor={red!50!black},
    citecolor={blue!80!black},
    urlcolor={blue!80!black}}
\usepackage{float}

\numberwithin{equation}{section}

\renewcommand{\phi}{\varphi}
\newcommand{\WR}{\mathcal{WR}}

\renewcommand{\d }{{\rm d} }

\newcommand{\h}{\hookrightarrow_h}
\newcommand{\G }{Cay (G, X\sqcup H)}
\newcommand{\C }{\mathcal C}

\newcommand{\e }{\varepsilon }

\renewcommand{\P }{\mathcal P}
\renewcommand{\kappa }{\varkappa}

\renewcommand{\ll }{\langle\hspace{-.7mm}\langle }
\newcommand{\rr }{\rangle\hspace{-.7mm}\rangle }
\renewcommand{\wr}{{\,\rm wr\,}}

\newcommand{\NN}{\mathbb N}

\newcommand{\M}{\mathcal M}
\newcommand{\Nn}{\mathcal N}

\newcommand{\A}{\mathcal A}

\newcommand{\D}{\mathcal D}

\newcommand{\Q}{\mathcal Q}
\newcommand{\R}{\mathcal R}

\newcommand{\Ss}{\mathcal S}
\newcommand{\sP}{\mathscr P}
\newcommand{\sU}{\mathscr U}
\newcommand{\sN}{\mathscr N}

\newcommand{\sZ}{\mathscr Z}
\newcommand{\sR}{\mathscr R}

\newcommand{\sM}{\mathscr M}

\newcommand{\sS}{\mathscr S}

\newcommand{\ZZ}{\mathbb Z}

\newtheorem{thm}{Theorem}[section]
\newtheorem*{thm*}{Theorem}
\newtheorem{cor}[thm]{Corollary}
\newtheorem{lem}[thm]{Lemma}

\newtheorem{prop}[thm]{Proposition}

\newtheorem{conj}[thm]{Conjecture}
\theoremstyle{definition}
\newtheorem{defn}[thm]{Definition}

\newtheorem{ex}[thm]{Example}
\newtheorem{fact}[thm]{Fact}
\theoremstyle{remark}
\newtheorem{rem}[thm]{Remark}

\let\OLDthebibliography\thebibliography
\renewcommand\thebibliography[1]{
  \OLDthebibliography{#1}
  \setlength{\parskip}{1.5pt}
  \setlength{\itemsep}{1.5pt plus 0.3ex}
}

\begin{document}

\title{Wreath-like products of groups and their von Neumann algebras I: W$^\ast$-superrigidity}

\author{Ionu\c t Chifan, Adrian Ioana, Denis Osin and Bin Sun}

\date{}

\maketitle

\begin{abstract}
We introduce a new class of groups called {\it wreath-like products}. These groups are close relatives of the classical wreath products and arise naturally in the context of group theoretic Dehn filling. Unlike ordinary wreath products, many wreath-like products have Kazhdan's property (T).  In this paper, we  prove that any group $G$ in a natural family of wreath-like products with property (T) is W$^*$-superrigid: the group von Neumann algebra $\text{L}(G)$ remembers the isomorphism class of $G$.  This allows us to provide the first examples (in fact, $2^{\aleph_0}$ pairwise non-isomorphic examples) of W$^*$-superrigid groups with property (T). \end{abstract}

\vspace{3mm}



\section{Introduction}


The von Neumann algebra $\text{L}(G)$ of a countable discrete group $G$ is defined as the weak operator closure of the complex group algebra $\mathbb CG$ acting on the Hilbert space $\ell^2G$ by left convolution \cite{MvN36}. If $G$ is infinite abelian, then L$(G)$ is isomorphic (via the Fourier transform)  to L$^{\infty}([0,1])$. However, understanding how the isomorphism class of L$(G)$ depends on $G$ for noncommutative groups is a notoriously challenging problem, which has been at the forefront of research in operator algebras since the creation of the field. This problem is typically studied when L$(G)$ is simple, i.e. a II$_1$ factor, which is equivalent to $G$ having infinite conjugacy classes of non-trivial elements (abbreviated {ICC}).

The classification problem for group von Neumann algebras was first considered by Murray and von Neumann in \cite{MvN43}. They proved that L$(G)$ is isomorphic to their hyperfinite II$_1$ factor for any locally finite ICC group $G$, but not for the free group $G=\mathbb F_2$.  Three decades later, Connes' celebrated classification of injective factors \cite{Co76} showed that, more generally, all II$_1$ factors arising from ICC amenable groups $G$ are isomorphic to the hyperfinite II$_1$ factor. On the other hand, nonamenable groups were used to provide large classes of nonisomorphic II$_1$ factors in  \cite{Sc63,Mc69,Co75}.

The first instance of rigidity for von Neumann algebras was discovered by Connes in 1980. He showed that L$(G)$ has countable fundamental and outer automorphism groups for any ICC group $G$ with Kazhdan's property (T) \cite{Co80}. Shortly after, Connes \cite{Co82}  proposed the following far-reaching rigidity conjecture.

\begin{conj}[Connes Rigidity Conjecture]\label{Cconj}
If $G$ and $H$ are \emph{ICC} property (T) groups such that $\emph{L}(G)\cong\emph{L}(H)$, then $G\cong H$.
\end{conj}

Supportive evidence for this conjecture was provided in the 1980s by results  in \cite{CJ85,Po86,CH89}.  In particular, Cowling and Haagerup \cite{CH89} proved that L$(G)\not\cong\text{L}(H)$, for any lattices $G<\text{Sp}(n,1)$, $H<\text{Sp}(m,1)$ for $n\not=m$. More recently, the existence of uncountably many nonisomorphic property (T) group factors was proved in \cite{Oz02}. In addition, Connes' rigidity conjecture was shown to hold up to countable classes in \cite{Po06a} (see also \cite{IPV10}).

In the past two decades, there has been striking progress in the classification of group II$_1$ factors due to Popa's discovery of deformation/rigidity theory. In \cite{Po04}, Popa proved that the class $\mathcal G$ of wreath product groups  $\mathbb Z \wr \Gamma$ with $\Gamma$ ICC property (T) satisfies the following version of Connes' rigidity conjecture:  if $\text{L}(G)\cong\text{L}(H)$, for $G,H\in\mathcal G$, then $G\cong H$. Subsequently, several other classes of groups satisfying this property were found, e.g., in \cite{Po06b,PV06,Io06,IM19}.

By a result in \cite{CJ85}, if an ICC group $G$ has property (T), then so does any other group $H$ such that $\text{L}(G)\cong\text{L}(H)$. Thus, Connes' rigidity conjecture is equivalent to asking if every ICC property (T) group $G$ is W$^*$-{\it superrigid} in the sense that  L$(G)\cong$ L$(H)$  implies $G\cong H$ for any group $H$ (see \cite[Section 3]{Po06a}).
The first class of W$^*$-superrigid groups was discovered by Popa, Vaes and the second author in \cite{IPV10}, where a large class of generalized wreath groups were shown to have this property. Later on,  additional examples of W$^*$-superrigid groups were found in  \cite{BV13,Be14,CI17,CD-AD20,CD-AD21}.

Despite the remarkable breadth of Popa's deformation/rigidity theory, classifying von Neumann algebras of property (T) groups remained a longstanding challenge.
The reason is that the presence of deformations, which is at the heart of Popa's theory, typically excludes property (T).
In particular, the well-known problem of finding at least one W$^*$-superrigid ICC group with property (T) remained open.
 This problem has circulated among the experts since the reformulation of Connes' rigidity conjecture as a superrigidity question in \cite[Section 3]{Po06a} (see also  \cite[Section 1]{IPV10}, \cite[Section 6.2]{Io18}, \cite[Problem R.5]{Pe20}, \cite[Section 5]{Houd}). It was the main focus of a 2018 workshop at the American Institute of Mathematics where it was explicitly posed by Vaes as \cite[Problem 2.1]{AIM18}.

In this paper, we obtain the first examples of W$^*$-superrigid groups with property (T). To state our main results, we need an auxiliary definition.

\begin{defn}\label{wlp}
Let $A$, $B$ be arbitrary groups, $I$ an abstract set, $B\curvearrowright I$ a (left) action of $B$ on $I$.
We say that a group $W$ is a \emph{wreath-like product} of groups $A$ and $B$ corresponding to the action $B\curvearrowright I$ if $W$
is an extension of the form
\begin{equation}\label{ext}
1\longrightarrow \bigoplus_{i\in I}A_i \longrightarrow  W \stackrel{\e}\longrightarrow B\longrightarrow 1,
\end{equation}
where $A_i\cong A$ and the action of $W$ on $A^{(I)}=\bigoplus_{i\in I}A_i$ by conjugation satisfies the rule
\begin{equation}
wA_iw^{-1} = A_{\e(w)\cdot i}
\end{equation}
for all $i\in I$. The map $\e\colon W\to B$ is called the {\it canonical homomorphism associated with the wreath-like structure of $W$}.

If the action $B\curvearrowright I$ is regular (i.e., free and transitive), we say that $W$ is a \emph{regular wreath-like product} of $A$ and $B$. The set of all wreath-like  products of groups $A$ and $B$ corresponding to an action $B\curvearrowright I$ (respectively, all regular wreath-like products) is denoted by $\WR(A, B\curvearrowright I)$ (respectively, $\WR(A,B)$).
\end{defn}

The notion of a wreath-like product generalizes the ordinary (restricted) wreath product of groups. Indeed, for any groups $A$ and $B$, we obviously have $A\,{\rm wr}\, B\in \WR(A,B)$. Conversely, it is not difficult to show that $W\cong A\,{\rm wr}\, B$ whenever the extension (\ref{ext}) splits.

We are now ready to state the main result of our paper.

\begin{thm}\label{superr}
Let $A$ be a nontrivial abelian group, $B$ a nontrivial ICC subgroup of a hyperbolic group. Let $B\curvearrowright I$ be an action such that, for every $i\in I$, $\emph{Stab}_B(i)$ is amenable. Suppose that $G\in\WR (A,B\curvearrowright I)$ is a group with property (T).

If $H$ is any countable group and $\theta\colon \emph{L}(G)^t\rightarrow \emph{L}(H)$ is a $*$-isomorphism for some $t>0$, then $G\cong H$ and $t=1$. Moreover, there exist a group isomorphism $\delta\colon G\rightarrow H$, a character $\eta\colon G\rightarrow\mathbb T$ and a unitary $w\in \emph{L}(H)$ such that $\theta(u_g)=\eta(g)wv_{\delta(g)}w^*$ for every $g\in G$, where $(u_g)_{g\in G}$ and $(v_h)_{h\in H}$ denote the canonical generating unitaries of $\emph{L}(G)$ and $\emph{L}(H)$, respectively.
\end{thm}

If $B$ is an ICC group and $B\curvearrowright I$ is an action with infinite orbits, then any group $G$ belonging to $\WR(A,B\curvearrowright I)$, for some group $A$, is ICC (see Lemma \ref{infinite} (b)). This implies that any $G$ as in Theorem \ref{superr} is ICC.

In Section \ref{Sec:WR}, we observe that examples of wreath-like products $G\in\WR (A,B\curvearrowright I)$ satisfying the assumptions of Theorem \ref{superr} naturally occur in the context of group theoretic Dehn filling. Here we mention just one particular case. For an element $h$ of a group $H$, we denote by $\ll h\rr$ the minimal normal subgroup of $H$ containing $h$.

\begin{thm}\label{Thm:HypWRintr}
Let $H$ be a torsion-free hyperbolic group and let $h\in H$ be a non-trivial element. For any sufficiently large $k\in\mathbb N$, the group $H/\ll h^k\rr $ is  hyperbolic, ICC, and we have $$H/[\ll h^k\rr, \ll h^k\rr]\in \WR (\ZZ, H/\ll h^k\rr\curvearrowright I),$$ where the action $H/\ll h^k\rr\curvearrowright I$ is transitive with finite cyclic stabilizers.
\end{thm}

In fact, we prove a more general result that also holds for hyperbolic groups with torsion (see Theorem \ref{Thm:HypWR}).  Combining Theorem \ref{Thm:HypWRintr} with Theorem \ref{superr}, we obtain many examples of non-trivial, W$^*$-superrigid, ICC groups with property (T).

\begin{cor}\label{MainCor} Let $H$ be a torsion-free hyperbolic group with property (T) and let $h\in H\setminus \{ 1\}$. For any sufficiently large $k\in\mathbb N$, the group $H/[\ll h^k\rr, \ll h^k\rr]$ is \emph{ICC} and \emph{W}$^*$-superrigid.
\end{cor}

Note that the group $H/[\ll h^k\rr, \ll h^k\rr]$ also has property (T) being a quotient of $H$. Torsion-free hyperbolic group with property (T) are abound; in fact, every property (T) group is a quotient of a torsion-free, hyperbolic, property (T) group  by a result of Cornulier \cite{Cor05}. Moreover, such groups are generic in the Gromov randomness model at density $1/3<d<1/2$ \cite{KK,Oli}. Finally, we mention some concrete linear examples.

\begin{ex}\label{Ex:RFHT}
Let $L$ be a uniform lattice in $Sp(n,1)$. By Selberg's lemma, there exists a finite index  torsion-free subgroup $H\le L$. Being a finite index subgroup of $L$, $H$ is also a uniform lattice in $Sp(n,1)$ and, therefore, is hyperbolic and has property (T).
\end{ex}

Wreath-like products obtained via Theorem \ref{Thm:HypWRintr} are not, in general, regular.  However, a little modification discussed in Section \ref{Sec:RWLP} allows us to construct $2^{\aleph_0}$ regular wreath-like products satisfying the assumptions of Theorem \ref{superr}. Thus, we obtain the following.

\begin{cor}\label{Cor:unc}
There exist $2^{\aleph_0}$ pairwise non-isomorphic, property (T), ICC groups that are \emph{W}$^\ast$-superrigid.
\end{cor}

The proof of Theorem \ref{superr} is given in Section \ref{SUPER} and  relies on a series of techniques and ideas from deformation/rigidity theory (e.g., \cite{Po01b,Po03,Po04,Po05,Po06b,Io06,OP07,CI08,PV09,Io10,IPV10,PV12,CIK13}). We refer the reader to the beginning of Section \ref{BER} for an outline of the proof of Theorem \ref{superr}. This proof plays property (T)
against properties of wreath-like product groups similar to properties of ordinary wreath products.

Although  wreath product groups $A\wr B$ are known to have remarkably rigid von Neumann algebras (see, e.g., \cite{Po03,Po04,Po06b,Io06,Io10, IPV10,IM19}), no wreath product $A\wr B$, with $A$ nontrivial abelian and $B$ torsion free, is W$^*$-superrigid by \cite[Theorem 1.2]{IPV10}.
It is precisely the presence of property (T) that  allows us to prove a stronger rigidity statement for wreath-like products $G\in\WR(A,B\curvearrowright  I)$. Using property (T) for $G$, rather than just for the quotient group $B$, is one of the main novelties of this paper.

\paragraph{Acknowledgments.} The first author has been supported by the NSF Grants  FRG-DMS-1854194 and DMS-2154637, the second author has been supported by NSF Grants FRG-DMS-1854074 and DMS-2153805, the third author has been supported by the NSF grant DMS-1853989, and the fourth author has received funding from the European Research Council (ERC) under the European Union’s Horizon 2020 research and innovation programme (Grant agreement No. 850930). We would like to thank the anonymous referee for several comments which helped improve the exposition.


\section{Wreath-like products of groups}\label{Sec:WR}


The main goal of this section is to prove  Theorem \ref{Thm:HypWR}, which is a more general (and more precise) version of Theorem \ref{Thm:HypWRintr}, and Corollary \ref{Cor:unc}. We begin by reviewing the necessary background.

\subsection{Hyperbolic groups and their generalizations}

We begin by recalling the necessary definitions. A group $G$ is {\it hyperbolic} if it is generated by a finite set $X$ and its Cayley graph $Cay (G,X)$ is a hyperbolic metric space. This definition is independent of the choice of a particular finite generating set $X$. A hyperbolic group is called \emph{elementary} if it contains a cyclic subgroup of finite index.

An isometric action of a group $G$ on a metric space $S$ is {\it acylindrical} if for every $\e>0$ there exist $R,N>0$ such that for every two points $x,y\in S$ with $\d (x,y)\ge R$, there are at most $N$ elements $g\in G$ satisfying
$$
\d(x,gx)\le \e \;\;\; {\rm and}\;\;\; \d(y,gy) \le \e.
$$

Every group has an acylindrical action on a hyperbolic space, namely the trivial action on the point. For this reason, we want to avoid elementary actions in the definition below. Recall that an action of a group $G$ on a hyperbolic space $S$ is \emph{non-elementary} if the limit set of $G$ on the Gromov boundary $\partial S$ has infinitely many points; for acylindrical actions, this condition is equivalent to the requirement that $G$ is not virtually cyclic and the action has infinite orbits {\cite[Theorem 1.1]{Osi16}}.

\begin{defn}\label{ahdef}
A group $G$ is \emph{acylindrically hyperbolic} if admits a non-elementary acylindrical action on a hyperbolic space.
\end{defn}

Every proper action is acylindrical. Therefore, every non-elementary hyperbolic group $G$ is acylindrically hyperbolic as witnessed by the proper action of $G$ on $Cay(G,X)$ for any finite generating set $X$. The class of acylindrically hyperbolic groups also includes many non-hyperbolic examples: mapping class groups of closed surfaces of non-zero genus, $Out(F_n)$ for $n\ge 2$, groups of deficiency at least $2$, most $3$-manifold groups, and many other examples. For more details we refer to the survey \cite{Osi18}.

Every acylindrically hyperbolic group contains a unique maximal finite normal subgroup \cite[Theorem 2.24]{DGO}). We denote it by $K(G)$. We will need the following result from \cite{DGO}.

\begin{thm}[{\cite[Theorem 2.35]{DGO}}]\label{Thm:HypICC}
An acylindrically hyperbolic group $G$ is \emph{ICC} if and only if $K(G)=\{ 1\}$.
\end{thm}

Let $H$ be a subgroup of a group $G$. We say that $X\subseteq G$ is a \emph{relative generating set} of $G$ (with respect to $H$) if $G=\langle X\cup H\rangle$. Associated to such a relative generating set is the Cayley graph $Cay (G, X\sqcup H)$, where the disjoint union means that for any element $a\in X\cap H$ and any vertex $g\in G$, there are two edges in $Cay (G, X\sqcup H)$ going from $g$ to $ga$: one is labeled by a copy of $a$ from $X$ and the other is labeled by a copy of $a$ from $H$. We denote by  $Cay(H,H)$ the Cayley graph of $H$ with respect to the generating set $H$ and naturally think of it as a (complete) subgraph of $\G$.

\begin{defn}\label{Def:HE}
A subgroup $H$ is {\it hyperbolically embedded} in a group $G$ if there exists a relative generating set $X$ of $G$ such that $Cay(G,X\sqcup H)$ is hyperbolic and for any $n\in \NN$, there are only finitely many elements $h\in H$ such that $h$ can be connected to $1$ in $\G$ by a path of length at most $n$ avoiding edges of $Cay(H,H)$.
\end{defn}

For example, it is easy to see that $H$ is hyperbolically embedded in the free product $H\ast \ZZ$, but not in the direct product $G=H\times \ZZ$.  For details, we refer to \cite{DGO}.

Recall that an element $g$ of a group $G$ acting on a hyperbolic space $S$ is \emph{loxodromic} if $g$ acts as a translation along a bi-infinite quasi-geodesic in $S$. If the action of $G$ on $S$ is acylindrical, this is equivalent to the requirement that $\langle g\rangle$ has unbounded orbits (see \cite[Theorem 1.1]{Osi16}). For example, if $G$ is a hyperbolic group acting on its Cayley graph with respect to a finite generating set, every infinite order element $g\in G$ is loxodromic. We will need the following.

\begin{thm}[{\cite[Theorem 6.8]{DGO}}]\label{Thm:Eg}
Let $G$ be an acylindrically hyperbolic group.
Every loxodromic element $g\in G$ is contained in a unique maximal virtually cyclic subgroup $E(g)$ such that $E(g)\h G$.
\end{thm}

\subsection{Wreath-like products associated to group theoretic Dehn filling}

Let $F$ be a free group and let $R=\ll r^k\rr$ for some $k\in \NN$ and $r\in F$, where $r$ is not a proper power. The Cohen-Lyndon theorem on relation modules of $1$-relator groups proved in \cite{CL} implies that $F/[R,R]\in \WR(\ZZ, F/R\curvearrowright I)$, where the action $F/R\curvearrowright I$ is transitive with stabilizers isomorphic to $\ZZ/k\ZZ$.

In this section, we show that this example has a natural analogue in the context of group theoretic Dehn filling. We begin by briefly surveying the relevant background. For more details, the reader is referred to \cite{DGO,Osi07,Sun}.

The classical Dehn surgery on a $3$-dimensional manifold consists of cutting off a solid torus, which may be thought of as ``drilling" along an embedded knot, and then gluing it back in a different way. The study of such transformations is partially motivated by the  Lickorish-Wallace theorem, which states that every closed orientable connected $3$-manifold can be obtained from the $3$-dimensional sphere by performing finitely many surgeries. The second part of the surgery, called {\it Dehn filling}, can be formalized as follows.

Let $M$ be a compact orientable $3$-manifold with toric boundary. Topologically distinct ways of attaching a solid torus to $\partial M$ are parameterized by free homotopy classes of unoriented essential simple closed curves in $\partial M$, called {\it slopes}. For a slope $s$, the corresponding   Dehn filling  $M(s)$ of $M$ is the manifold obtained from $M$ by attaching a solid torus to $\partial M$ so that the meridian of the torus goes to a simple closed curve of the slope $s$. The fundamental theorem due to Thurston \cite[Theorem 1.6]{Th} asserts that if $M\setminus\partial M$ admits a finite volume hyperbolic structure, then $M(s)$ is hyperbolic for all but finitely many slopes. Note that, in the settings of Thurston's theorem, we can think of $s$ as an element of $\pi_1(\partial M)\le \pi_1(M)$ and, by the Seifert--van Kampen theorem, we have
\begin{equation}\label{Eq:pi1}
\pi_1(M(s))=\pi_1(M)/\ll s\rr.
\end{equation}

In group-theoretic settings, the role of the pair $\partial M\subset M$ is played by a pair of groups $H\le G$ and the existence of a finite volume hyperbolic structure on $M\setminus\partial M$ translates to the property that $H$ is hyperbolically embedded in $G$. Equation (\ref{Eq:pi1}) suggests that the process of attaching a solid torus to $M$ must correspond to taking the quotient of $G$ modulo the normal closure of an element $s\in H$. In fact, we can consider even more general quotients.

For a group $G$ and a subset $S\subseteq G$, we denote by $\ll S\rr$ the \emph{normal closure} of $S$ in $G$; that is $\ll S\rr$ is the smallest normal subgroup of $G$ containing $S$. The following result can be thought of as the algebraic analogue of Thurston's theorem.

\begin{thm}[Dahmani--Guirardel--Osin] \label{Thm:DF}
Let $G$ be a group, $H$ a hyperbolically embedded subgroup of $G$. There exists a finite subset $\mathcal F\subseteq H\setminus\{ 1\}$ such that for any $N\lhd H$ satisfying $N\cap \mathcal F=\emptyset$, the natural map $H/N\to G/\ll N\rr$ is injective and $H/N\h G/\ll N\rr$ (by abuse of notation, we identify $H/N$ with its image in $G/\ll N\rr$).
\end{thm}

 For relatively hyperbolic groups, this theorem was obtained in \cite{Osi07} (an independent proof for torsion-free groups was given by Groves and Manning in \cite{GM}) and the general version was proved in \cite[Theorem 2.27]{DGO}.

We now state the main result of this section. Recall that, for a loxodromic element $g$ of an acylindrically hyperbolic group $G$, $E(g)$ denotes the maximal virtually cyclic subgroup of $G$ containing $g$.

\begin{thm}\label{Thm:HypWR}
Let $G$ be an acylindrically hyperbolic group, $g\in G$ a loxodromic element. Let $d$ be a natural number such that $\langle g^d\rangle \lhd E(g)$. For every sufficiently large $k\in \NN$ divisible by $d$, the following hold.

\begin{enumerate}
\item[(a)] We have
\begin{equation}\label{Eq:GWR}
G/[\ll g^{k}\rr, \ll g^{k}\rr] \in \WR (\ZZ, G/\ll g^k\rr \curvearrowright I),
\end{equation}
where $I$ is the set of cosets $G/E(g)\ll g^k\rr$ and the action is by left multiplication. In particular, the action $G/\ll g^k\rr \curvearrowright I$ is transitive.
\item[(b)] The stabilizers of the action $G/\ll g^k\rr \curvearrowright I$ are isomorphic to $E(g)/\langle g^k\rangle$.
\item[(c)] If $G$ is \emph{ICC}, then so is $G/\ll g^k\rr$.
\item[(d)] \emph{(Olshanskii, \cite{Ols})} If $G$ is hyperbolic, then so is $G/\ll g^k\rr$.
\end{enumerate}
\end{thm}

Note that Theorem \ref{Thm:HypWRintr} is a particular case of Theorem \ref{Thm:HypWR}. Indeed, if $G$ is torsion-free, then so is $E(g)$. Every torsion-free virtually cyclic group is cyclic (see, for example, \cite[Lemma 2.5]{JMN}). Therefore, we can take $d=1$ and Theorem \ref{Thm:HypWRintr} follows.

For a group $R$, we denote by $R^{ab}$ its abelianization; that is, $R^{ab}=R/[R, R]$. The main ingredient of the proof of Theorem \ref{Thm:HypWR} is the following result, which was stated and proved in \cite{Sun} using a slightly different terminology.

\begin{thm}[Sun \cite{Sun}]\label{Lem:SunComm}
Let $G$ be a group, $H$ a hyperbolically embedded subgroup of $G$. There exists a finite subset $\mathcal F\subseteq H\setminus\{ 1\}$ such that for any $N\lhd H$ satisfying $N\cap \mathcal F=\emptyset$, we have
\begin{equation}\label{Eq:Sun}
G/[\ll N\rr, \ll N\rr] \in \WR (N^{ab}, G/\ll N\rr\curvearrowright I),
\end{equation}
where $I=G/H\ll N\rr$, the action of $G/\ll N\rr$ on $I$ is by left multiplication, and stabilizers of elements of $I$ are isomorphic to $H/N$.
\end{thm}

\begin{proof}[On the proof]
Since our terminology is different from the one used by Sun, we explain how to derive the theorem from the main result of \cite{Sun}. Consider the short exact sequence
$$
1\longrightarrow \ll N\rr^{ab} \longrightarrow  G/[\ll N\rr, \ll N\rr] \longrightarrow G/\ll N\rr \longrightarrow 1.
$$
By \cite[Corollary 2.8]{Sun}, we have an isomorphism of $G/\ll N\rr$-modules
\begin{equation}\label{Eq:SunIso}
\ll N\rr^{ab}\cong {\rm Ind}_{H/N}^{G/\ll N\rr} N^{ab},
\end{equation}
where the actions of $G/\ll N\rr $ on $\ll N\rr ^{ab}$ and $H/N$ on $N^{ab}$ are induced by conjugation. Note that we can assume $H/N$ to be a subgroup of $G/\ll N\rr $ by Theorem \ref{Thm:DF}. The standard description of the algebraic structure of the induced module (see, for example, \cite[Ch. III, Proposition 5.1]{Bro}) and (\ref{Eq:SunIso}) easily imply (\ref{Eq:Sun}).
\end{proof}

We are now ready to prove the main result of this section.

\begin{proof}[Proof of Theorem \ref{Thm:HypWR}]
By Theorem \ref{Thm:Eg}, we have $E(g)\h G$. Let $\mathcal F$ be a finite subset of $E(g)\setminus \{ 1\}$ such that the conclusions of Theorem \ref{Thm:DF} and Theorem \ref{Lem:SunComm} simultaneously hold true for $H=E(g)$ and any $N\lhd H$ satisfying $N\cap \mathcal F=\emptyset$.

Since $k$ is divisible by $d$, we have $\langle g^k\rangle \lhd E(g)$. Further, by taking $k$ sufficiently large, we can ensure the condition $\langle g^k\rangle \cap \mathcal F=\emptyset$. By Theorem \ref{Lem:SunComm}, we have (\ref{Eq:GWR}) where $I$ is the set of cosets $G/E(g)\ll g^k\rr$ and the action $G/\ll g^k\rr \curvearrowright I$ is by left multiplication. In particular, the action $G/\ll g^k\rr \curvearrowright I$ is transitive and the stabilizers are isomorphic to $E(g)\ll g^k\rr/\ll g^k\rr \cong E(g)/ \langle g^k\rangle $ by Theorem \ref{Thm:DF}. This gives parts (a) and (b) of the theorem.

The proof of (c) makes use of a more general Dehn filling procedure with multiple hyperbolically embedded subgroups and some other results about acylindrically hyperbolic groups. Since these results are not used anywhere else in our paper, we do not discuss them in detail. Instead, we refer the reader to the appropriate places in the relevant papers.

Assume that $G$ is ICC. By Theorem \ref{Thm:HypICC}, we have $K(G)=\{ 1\}$. Let $X$ be a relative generating set of $G$ with respect to $H=E(g)$ satisfying the conditions listed in Definition \ref{Def:HE}. Combining Proposition 5.14 and Corollary 3.12 from \cite{AMS}, we obtain an element $h\in G$ acting loxodromically on $Cay(G, X\sqcup H)$ such that $E(h)=\langle h\rangle $ and the collection of subgroups $\{ E(g), E(h)\}$ is hyperbolically embedded in $G$ (for the definition of a hyperbolically embedded collection of subgroups, see \cite[Definition 4.25]{DGO}).

By \cite[Theorem 7.19]{DGO}, which is a more general version of Theorem \ref{Thm:DF}, we can choose a finite subset $F\subseteq E(g)\setminus \{1\}$ so that for any $N\lhd E(g)$ satisfying $N\cap \mathcal F=\emptyset$, the natural maps $E(g)/N\to G/\ll N\rr$ and $E(h)\to G/\ll N\rr$ are injective and the collection $\{ E(g)/N, E(h)\} $ is hyperbolically embedded in $G/\ll N\rr$ (by abuse of notation, we identify $E(g)/N$ and $E(h)$ with their isomorphic images in $G/\ll N\rr$). In particular, this is true for  $N=\langle g^k\rangle $ for all sufficiently large $k$ divisible by $d$. By \cite[Proposition 2.10]{DGO}, torsion-free hyperbolically embedded subgroups are malnormal; therefore, the infinite cyclic subgroup $E(h)$ is malnormal in $G/\ll g^k\rr$. This easily implies that $K(G/\ll g^k\rr)=\{ 1\}$, which is equivalent to $G/\ll g^k\rr$ being ICC by  Theorem \ref{Thm:HypICC}.

Finally, part (d) was proved for all sufficiently large $k$ divisible by $d$ in \cite[Theorem 3]{Ols}.
\end{proof}


\subsection{Regular wreath-like products}\label{Sec:RWLP}


In this section, we use Theorem \ref{Thm:HypWR} to construct (uncountably many) regular wreath-like products satisfying the assumptions of Theorem \ref{superr}.  We begin with a lemma that allows us to obtain regular wreath-like products from non-regular ones. We restrict ourselves to transitive actions for notational simplicity; the generalization to arbitrary actions is straightforward.

\begin{lem}\label{WRsubgr}
Let $A$, $B$ be arbitrary groups, $B\curvearrowright I$ a transitive action of $B$ on a set $I$, and let $W\in \WR(A,B\curvearrowright I)$. Further, let $D\le B$ and let $V\le W$ denote the full preimage of $D$ under the canonical homomorphism $W\to B$. Suppose that  the induced action $D\curvearrowright I$ is free and let $O$ denote the set of $D$-orbits in $I$. Then $V\in \WR (C,D)$, where $C$ is the direct sum of $|O|$-many isomorphic copies of $A$.
\end{lem}

\begin{proof}
Throughout the proof, we use the notation introduced in Definition \ref{wlp}. Fix some $i_0\in I$. Since the action $B\curvearrowright I$ is transitive, there exists $T\subseteq B$ such that for every orbit $o\in O$, there is a unique $t\in T$ such that $ti_0\in o$. For every $d\in D$, we define
$$
I_d=\{ dti_0\mid t\in T\}\subseteq I \;\;\; {\rm and}\;\;\; C_d=\left\langle \bigcup_{i\in I_d} A_i\right\rangle \le A^{(I)}.
$$
Note that the equality $dti_0=d^\prime t^\prime i_0$ implies $t=t^\prime$ since otherwise $dti_0$ and $d^\prime t^\prime i_0$ belong to distinct $D$-orbits. Since the action of $D$ on $I$ is free, we obtain $d=d^\prime$. Therefore, $|I_d|=|T|=|O|$ for all $d\in D$ and $I_d\cap I_{d^\prime}=\emptyset$. The former equality implies that $C_d$ is the direct sum of $|O|$-many isomorphic copies of $A$ for every $d$. Further, the decomposition $I=\bigsqcup_{d\in D} I_d$ yields the decomposition $A^{(I)}=\bigoplus_{d\in D} C_d$. Finally, for every $v\in V$, we have
$$
vC_dv^{-1}=v\left\langle \bigcup_{t\in T} A_{dti_0}\right\rangle v^{-1}=\left\langle \bigcup_{t\in T} v A_{dti_0}v^{-1}\right\rangle=\left\langle \bigcup_{t\in T}  A_{\e(v)dti_0}\right\rangle = C_{\e(v)d}
$$
and the result follows.
\end{proof}

In the next result, we use the notation of Theorem \ref{Thm:HypWR}.

\begin{cor}\label{CorAwrG0}
Let $G$, $g$, and $k$ satisfy the assumptions of Theorem \ref{Thm:HypWR}. Suppose, in addition, that $G_0$ is a normal subgroup of $G$ such that $G_0\cap E(g)=\langle g^k\rangle $. We keep the notation $\ll g^k\rr$ for the normal closure of $g^k$ in $G$.  Then $\ll g^k\rr \le G_0$ and $G_0/[\ll g^{k}\rr, \ll g^{k}\rr]\in \WR (A, G_0/\ll g^k\rr)$, where $A$ is free abelian. Moreover, if $|G:G_0|=\infty$, then $A$ is of countably infinite rank.
\end{cor}

\begin{proof}
Clearly, we have $\ll g^k\rr \le G_0$ since $g^k\in G_0$ and $G_0\lhd G$. Let $W=G/[\ll g^{k}\rr, \ll g^{k}\rr]$ and let $\e\colon W\to G/\ll g^k\rr$ be the canonical homomorphism associated with the wreath-like structure of $W$ described in parts (a) and (b) of Theorem \ref{Thm:HypWR}. Further, let $V=G_0/[\ll g^{k}\rr, \ll g^{k}\rr]$ denote the image of $G_0$ in $W$. Since $G_0\cap E(g)=\langle g^k\rangle $ and $G_0\lhd G$, the induced action of $\e(V)=G_0/\ll g^k\rr$ on the set $I=G/E(g)\ll g^k\rr$ is free and Lemma \ref{WRsubgr} applies. It remains to note that if $|G:G_0|=\infty$, then the number of $\e(V)$-orbits in $I$ is infinite since $|\e(W):\e(V)|=|W:V|=|G:G_0|$ and all stabilizers of the action $\e(W)\curvearrowright I$ are finite by Theorem \ref{Thm:HypWR} (b).
\end{proof}

We will also need the following.

\begin{lem}[{\cite[Corollary 1.2]{BO}}]\label{Lem:BO}
For any finitely presented torsion-free group $Q$, there exists a short exact sequence $1\to N\to G\to Q\to 1$, where $G$ is torsion-free hyperbolic and $N$ is a non-trivial group with property (T).
\end{lem}

We are now ready to construct the first examples of regular wreath-like products satisfying the assumptions of Theorem \ref{superr}.

\begin{prop}\label{Prop:ZwrB}
There exists a property (T) group $V\in \WR(\mathbb Z^\infty, B)$, where $B$ is a non-trivial, ICC subgroup of a hyperbolic group, and $\mathbb Z^\infty$ denotes the free abelian group of countably infinite rank.
\end{prop}

\begin{proof}
Let $S$ be any finitely presented, residually finite, torsion-free group with property (T) (e.g., we can take $S=H$, where $H$ is the group constructed in Example \ref{Ex:RFHT}). By Lemma \ref{Lem:BO}, there exists a short exact sequence $$1\to N\to G\stackrel{\gamma}\to S\times \mathbb Z\to 1,$$ where $G$ is torsion-free hyperbolic and $N$ is a non-trivial normal subgroup of $G$ with  property (T). Let $g$ be an element of $G$ such that $\gamma(g)\in S\setminus\{ 1\}$. It is well-known that every torsion-free virtually cyclic group is cyclic (see, for example, \cite{Sta}). Thus, replacing $g$ with a generator of $E(g)$ if necessary, we can assume that $E(g)=\langle g\rangle$.

By Theorem \ref{Thm:HypWR} (applied in the particular case $d=1$), there is $K\in \mathbb N$ such that, for every $k\ge K$, we have
\begin{equation}\label{Eq:W_k}
W = G/[\ll g^k\rr, \ll g^k\rr ] \in \mathcal{WR}(\mathbb Z, G/\ll g^k\rr \curvearrowright I),
\end{equation}
where $G/\ll g^k\rr $ acts on $I=G/E(g)\ll g^k\rr$ by left multiplication, and conditions (b)--(d) of the theorem hold. Since $S$ is residually finite, we can find a finite index normal subgroup $S_0\lhd S$ such that
\begin{equation}\label{Eq:g^i}
\gamma(g^i)\notin S_0\times \{ 1\}\;\;\; \text { for all }\, 1\le i\le K.
\end{equation}
Let $G_0\lhd G$ be the full preimage of $S_0\times \{ 1\} \le S\times \mathbb Z$ under $\gamma$. By the choice of $S_0$ (see (\ref{Eq:g^i})), we have
$G_0\cap E(g)=\langle g^k\rangle$ for some $k\ge K$. Let $W$ be the group defined by (\ref{Eq:W_k}) and let $V$ be the image of $G_0$ in $W$. Note that $|G:G_0|= |(S\times \mathbb Z):(S_0\times\{1\})| =\infty$. By Corollary \ref{CorAwrG0}, we have $V\in \WR(\mathbb Z^\infty, G_0/\ll g^k\rr)$, where $\ll g^k\rr$ is the normal closure of $g^k$ in $G$ (not in $G_0$).

Recall that the class of groups with property (T) is closed under extensions and taking subgroups of finite index. Thus, the group $G_0$ has property (T) being an extension of $N$ by a finite index subgroup $S_0$ of $S$. Hence, $V$ has property (T). By part (d) of Theorem \ref{Thm:HypWR}, the group $B=G/\ll g^k \rr$ is hyperbolic and ICC. To complete the proof, it remains to show that $D=G_0/\ll g^k\rr$ is ICC.

To this end, we first note that $D\ne \{1\}$ since otherwise $V=\mathbb Z^\infty$, which contradicts the fact that $V$ has property (T). Further, since $D\lhd B$ and $B$ is ICC, $D$ must be infinite. Combining this with property (T), we conclude that $D$ cannot be virtually cyclic. Thus, $D$ is a non-elementary subgroup of the hyperbolic group $B$. By Theorem \ref{Thm:HypICC}, it suffices to show that $K(D)=\{1\}$. Note that $K(D)$ is characteristic in $D$ and, therefore, normal in $B$. Since $B$ is ICC, we have $K(D)\le K(B)=\{ 1\}$ and the desired result follows.
\end{proof}

To obtain an uncountable family of regular wreath-like products satisfying the assumptions of Theorem \ref{superr}, we will combine Proposition \ref{Prop:ZwrB} with the following.

\begin{lem}\label{A/N}
Let $A$, $B$ be any groups, $W\in \WR(A,B)$. We identify $A$ with the subgroup $A_1$ of the base $\bigoplus_{b\in B}A_b\le W$. For any $N\lhd A$, we have $W/\ll N\rr \in \WR (A/N, B)$.
\end{lem}

\begin{proof}
For every $b\in B$, we define $N_b=uNu^{-1}$, where $u$ is an element of $W$ such that
\begin{equation}\label{eq:eu=b}
\e(u)=b.
\end{equation}
Note that the subgroup $N_b$ is independent of the choice of a particular element $u\in W$ satisfying (\ref{eq:eu=b}). Indeed, if $v\in W$ is another element such that $\e(v)=b$, then $u^{-1}v\in A^{(B)}$. Obviously, $N\lhd A^{(B)}$. Therefore, $(u^{-1}v)N(u^{-1}v)^{-1}=N$, which implies $uNu^{-1}=vNv^{-1}$.

It is easy to see that $\ll N\rr=\bigoplus_{b\in B} N_b$. Hence, $W/\ll N\rr$ splits as
$$
1\longrightarrow \bigoplus_{b\in B}A_b/N_b \longrightarrow  W/\ll N\rr \stackrel{\delta}\longrightarrow B\longrightarrow 1,
$$
where $\delta $ is induced by the canonical homomorphism $W\to B$. It remains to note that we have $w(A_b/N_b) w^{-1}=A_{\delta(w)b}/N_{\delta(w)b}$ for all $w\in W/\ll N\rr$ and $b\in B$.
\end{proof}

We are now ready to prove the result announced in the introduction.

\begin{proof}[Proof of Corollary \ref{Cor:unc}]
Let $V\in \WR(\mathbb Z^\infty, B)$ be the group provided by Proposition \ref{Prop:ZwrB}. For every infinite set of primes $\mathcal P=\{p_1, p_2, \ldots\}$, the application of Lemma \ref{A/N} to the  subgroup $p_1\mathbb Z \oplus p_2\mathbb Z \oplus\ldots \lhd \mathbb Z^\infty$ yields a group $V_\mathcal P\in \WR \left(\bigoplus _{p\in \mathcal P}\mathbb Z/p\mathbb Z, B\right)$. The group $V_\mathcal P$ has property (T) being a quotient of $V$. Proposition \ref{Prop:ZwrB} guarantees that $B$ is a non-trivial, ICC subgroup of a hyperbolic group.

It remains to note that $V_\mathcal P\not \cong V_{\mathcal P^\prime}$ whenever $\mathcal P\ne \mathcal P^\prime$. Indeed, $B$ does not contain any non-trivial, normal, abelian subgroups since it is hyperbolic and ICC. Therefore, $V_\mathcal P$ has a unique maximal abelian normal subgroup isomorphic to a direct sum of copies of $\bigoplus _{p\in \mathcal P}\mathbb Z/p\mathbb Z$. Thus, the maximal abelian normal subgroup of $V_\mathcal P$ contains an element of prime order $p$ if and only if $p\in \mathcal P$ and the result follows.
\end{proof}


\section{Preliminaries on von Neumann algebras}\label{prelimvN}



\subsection{Tracial von Neumann algebras}\label{tracialvN}


We start by recalling some terminology and constructions involving tracial von Neumann algebras and refer the reader to \cite{AP} for more information.

A {\it tracial von Neumann algebra} is a pair $(\M,\tau)$ consisting of a von Neumann algebra $\M$  and a {\it trace} $\tau$, i.e., a normal  faithful tracial state $\tau\colon\M\rightarrow\mathbb C$. For $x\in \M$, we denote by $\|x\|$ the operator norm of $x$ and by $\|x\|_2=\tau(x^*x)^{1/2}$ its (so-called)  $2$-norm. We denote by L$^2(\M)$ the Hilbert space obtained as the closure of $\M$ with respect to the $2$-norm, by $\sU(\M)$ the group of {\it unitaries}  of $\M$, and by $(\M)_1=\{x\in \M\mid \|x\|\leq 1\}$ the {\it unit ball} of $\M$. 
We always assume that $\M$ is {\it separable}, i.e., that L$^2(\M)$ is a separable Hilbert space. We denote by ${Aut}(\M)$ the group of $\tau$-preserving automorphisms of $\M$.  For $u\in\sU(\M)$,  the {\it inner} automorphism $\text{Ad}(u)$ of $\M$ is given by $\text{Ad}(u)(x)=uxu^*$.
By von Neumann's bicommutant theorem, for any set $X\subset \M$ closed under adjoint, $X''\subset \M$ is the smallest von Neumann subalgebra which contains $X$.
 For a set $I$, we denote by $(\M^I,\tau)$ the tensor product of tracial von Neumann algebras $\,\overline{\otimes}\,_{i\in I}(\M,\tau)$. Given a subset $J\subset I$, we view $\M^J$ as a subalgebra of $\M^I$ by identifying it with $(\,\overline{\otimes}\,_{i\in J}\M)\,\overline{\otimes}\,(\,\overline{\otimes}\,_{i\in I\setminus J}1)$.

An {\it $\M$-bimodule} is a Hilbert space $\mathcal H$ equipped with two normal $*$-homomorphisms $\pi_1\colon\M\rightarrow\mathbb B(\mathcal H)$ and $\pi_2\colon\M^{\text{op}}\rightarrow\mathbb B(\mathcal H)$ whose images commute.
We write $x\xi y=\pi_1(x)\pi_2(y^{\text{op}})\xi$ for $\xi\in\mathcal H$ and define a $*$-homomorphism $\pi_{\mathcal H}\colon\M\otimes_{\text{alg}} \M^{\text{op}}\rightarrow\mathbb B(\mathcal H)$ by letting $\pi_{\mathcal H}(x\otimes y^{\text{op}})=\pi_1(x)\pi_2(y^{\text{op}})$. 
Examples of bimodules include the {\it trivial} $\M$-bimodule L$^2(\M)$ and the {\it coarse} $\M$-bimodule $\text{L}^2(\M)\otimes \text{L}^2(\M)$.
We say that $\mathcal H$ is {\it weakly contained} in  another $\M$-bimodule $\mathcal K$ and write $\mathcal H\subset_{\text{weak}}\mathcal K$ if $\|\pi_{\mathcal H}(T)\|\leq\|\pi_{\mathcal K}(T)\|$, for every $T\in \M\otimes_{\text{alg}} \M^{\text{op}}$.

Let $\Q\subset \M$ be a von Neumann subalgebra, which we always assume to be unital. 
We denote by $\Q'\cap \M=\{\text{$x\in \M\mid xy=yx$, for all $y\in \Q$}\}$ the {\it relative commutant} of $\Q$ in $\M$, and by $\sN_\M(\Q)=\{u\in\sU(\M)\mid u\Q u^*=\Q\}$ the {\it normalizer} of $\Q$ in $\M$. The {\it center} of $\M$ is given by $\sZ(\M)=\M'\cap \M$.
We say that $\Q$ is {\it regular} in $\M$ if $\sN_\M(\Q)''=\M$. If $\Q\subset \M$ is regular and maximal abelian, we call it a {\it Cartan subalgebra}.

{\it Jones' basic construction} $\langle \M,e_\Q\rangle$ is defined as the von Neumann subalgebra of $\mathbb B(\text{L}^2(\M))$ generated by $\M$ and the orthogonal projection $e_\Q$ from L$^2(\M)$ onto L$^2(\Q)$. The basic construction $\langle \M,e_\Q\rangle$ has a faithful semi-finite trace given by $\text{Tr}(xe_\Q y)=\tau(xy)$, for every $x,y\in \M$. We denote by L$^2(\langle \M,e_\Q\rangle)$ the associated Hilbert space and endow it with the natural $\M$-bimodule structure. We also denote by $E_\Q\colon\M\rightarrow \Q$ the unique $\tau$-preserving {\it conditional expectation} onto $\Q$.

The tracial von Neumann algebra $(\M,\tau)$  is called {\it amenable} if there exists a sequence $\xi_n\in \text{L}^2(\M)\otimes \text{L}^2(\M)$ such that $\langle x\xi_n,\xi_n\rangle\rightarrow\tau(x)$ and $\|x\xi_n-\xi_nx\|_2\rightarrow 0$, for every $x\in \M$.

Let $\P\subset p\M p$ be a von Neumann subalgebra.
Following Ozawa and Popa \cite[Section 2.2]{OP07} we say that $\P$ is  {\it amenable relative to $\Q$ inside $\M$} if there exists a sequence $\xi_n\in \text{L}^2(\langle \M,e_\Q\rangle)$ such that $\langle x\xi_n,\xi_n\rangle\rightarrow\tau(x)$, for every $x\in p\M p$, and $\|y\xi_n-\xi_ny\|_2\rightarrow 0$, for every $y\in \P$.
We say that $\P$ is {\it strongly nonamenable relative to $\Q$ inside $\M$} if there exist no nonzero projection $p'\in \P'\cap p\M p$ such that $\P p'$ is amenable relative to $\Q$ inside $\M$.

\begin{rem}\label{TT} Assume that $P$ is amenable relative to $Q$ inside $M$. By the proof of \cite[Theorem 2.1]{OP07}, in the definition of relative amenability we may take $\xi_n=\zeta_n^{1/2}$, for positive $\zeta_n\in\text{L}^1(\langle \M,e_\Q\rangle)$.   Thus,  $\langle \xi_n x,\xi_n\rangle=\text{Tr}(\zeta_nx)=\langle x\xi_n,\xi_n\rangle\rightarrow\tau(x)$, for all $x\in \M$. Using a convexity argument (see the proof of \cite[Lemma 13.3.11]{AP}), we  find $\eta_n\in \text{L}^2(\langle \M,e_\Q\rangle)^{\oplus\infty}$ such that $\|\langle\cdot\eta_n,\eta_n\rangle-\tau(\cdot)\|\rightarrow 0$, $\|\langle\eta_n\cdot,\eta_n\rangle-\tau(\cdot)\|\rightarrow 0$ and $\|y\eta_n-\eta_ny\|_2\rightarrow 0$, for all $y\in \P$ .
\end{rem}

Following \cite[Proposition 4.1]{Po01b}, we say that $\Q\subset \M$ has the {\it relative property (T)} if for every $\varepsilon>0$, we can find a finite set $F\subset \M$ and $\delta>0$ such that if $\mathcal H$ is an $\M$-bimodule and $\xi\in\mathcal H$ satisfies $\|\langle\cdot\xi,\xi\rangle-\tau(\cdot)\|\leq\delta,\|\langle\xi\cdot,\xi\rangle-\tau(\cdot)\|\leq\delta$ and $\|x\xi-\xi x\|\leq\delta$, for every $x\in F$, then there exists $\eta\in\mathcal H$ such that $\|\eta-\xi\|\leq\varepsilon$ and $y\eta=\eta y$, for every $y\in \Q$.


\subsection {Intertwining-by-bimodules}


We recall from  \cite [Theorem 2.1, Corollary 2.3]{Po03} Popa's {\it intertwining-by-bimodules} theory.

\begin{thm}[\cite{Po03}]\label{corner} Let $(\M,\tau)$ be a tracial von Neumann algebra, $\P\subset p\M p, \Q\subset q\M q $ be von Neumann subalgebras and $\mathscr G\subseteq \mathscr U(\P)$ be any subgroup which generates $\P$ as a von Neumann algebra. Consider the following conditions.
\begin{enumerate}
\item[(a)] There exist projections $p_0\in \P, q_0\in \Q$, a $*$-homomorphism $\theta\colon p_0\P p_0\rightarrow q_0\Q q_0$  and a nonzero partial isometry $v\in q_0\M p_0$ such that $\theta(x)v=vx$, for all $x\in p_0\P p_0$.

\item[(b)] There is no sequence  $u_n\in\mathscr G$ satisfying $\|E_\Q(x^*u_ny)\|_2\rightarrow 0$, for all $x,y\in p\M$.


\item[(c)] There exists a nonzero element $a\in \P'\cap p\langle \M,e_\Q\rangle p$ such that $a\geq 0$ and $\emph{Tr}(a)<\infty$.
\end{enumerate}
Conditions (a) and (b) are equivalent in general, and (a), (b) and (c) are equivalent if $q=1$.
\end{thm}

If (a) or (b) hold true,  we write $\P\prec_{\M}\Q$ and say that {\it a corner of $\P$ embeds into $\Q$ inside $\M$.}
If $\P p'\prec_{\M}\Q$, for any nonzero projection $p'\in \P'\cap p\M p$, we write $\P\prec^{\text s}_{\M}\Q$.


\subsection{Cocycle superrigidity}\label{Sec:CSR}


In this subsection, we record a cocycle superrigidity result that will be needed to prove Theorem \ref{superr}. 
Let $G$ be a countable group.  By a {\it trace preserving} action $G\curvearrowright^{\sigma}(\P,\tau)$ we mean a homomorphism $\sigma\colon G\rightarrow{Aut}(\P)$, where $(\P,\tau)$ is a tracial von Neumann algebra.
A {\it $1$-cocycle for $\sigma$} is a map $w\colon G\rightarrow\sU(\P)$ such that $w_{gh}=w_g\sigma_g(w_h)$, for every $g,h\in G$. Any character $\eta\colon G\rightarrow\mathbb T$ gives a (trivial) $1$-cocycle for $\sigma$.
Two cocycles $w,w'\colon G\rightarrow\sU(\P)$ are called {\it cohomologous} if there is $u\in\sU(\P)$ such that $w'_g=u^*w_g\sigma_g(u)$, for every $g\in G$.
Let $p\in \P$ be a projection. A {\it generalized $1$-cocycle for $\sigma$ with support projection $p$} is a map $w\colon G\rightarrow \P$ such that $w_gw_g^*=p,w_g^*w_g=\sigma_g(p)$ and $w_{gh}=w_g\sigma_g(w_h)$, for every $g,h\in G$.

\begin{ex}\label{exmpl} We continue by recording several examples of trace preserving actions.
\begin{enumerate}
\item[(a)] Let $G\curvearrowright I$ be an action on a  countable set $I$ and $(\P,\tau)$ a tracial von Neumann algebra.
The {\it generalized Bernoulli action} $G\curvearrowright^\sigma (\P^I,\tau)$ associated to $G\curvearrowright I$ is given by $\sigma_g(x)=\otimes_{i\in I}x_{g^{-1}\cdot i}$, for all $g\in G$ and $x=\otimes_{i\in I}x_i\in \P^I$ with $\{i\in I\mid x_i\not=1\}$ finite.
 If $i\in I$,  we let $\text{Stab}_G(i)$ be the  stabilizer of $i$ in $G$ and denote  $\P^{\{i\}}$ by $\P^i$.

\item[(b)] Let $K<G$ be a subgroup and $K\curvearrowright^{\sigma} (\P,\tau)$ be a trace preserving action.
Let $\varphi\colon G/K\rightarrow G$ such that $\varphi(h)K=h$, for every $h\in G/K$. Define $c\colon  G\times G/K\rightarrow K$ by $c(g,h)=\varphi(gh)^{-1}g\varphi(h)$, for $g\in G$ and $h\in G/ K$. For $h\in G/K$, let $\rho_h\colon \P\rightarrow \P^{G/K}$ be the embedding given by identifying $\P$ with $\P^{h}$. The {\it co-induced action} $G\curvearrowright^{\widetilde\sigma}\P^{G/K}$ is given by the formula $\widetilde\sigma_g(\rho_h(x))=\rho_{gh}(\sigma_{c(g,h)}(x))$, for all $g\in G,h\in G/K$ and $x\in \P$.


\item[(c)] Let $G\curvearrowright I$ be an action on a countable set $I$ and $(\P,\tau)$  a tracial von Neumann algebra. Following Krogager and Vaes \cite[Definition 2.5]{KV15}, we say that a trace preserving action $G\curvearrowright^{\sigma}(\P^I,\tau)$ is {\it built over $G\curvearrowright I$}  if it satisfies $\sigma_g(\P^i)=\P^{g\cdot i}$, for every $g\in G$ and $i\in I$. Let $J\subset I$ be a set which meets each $G$-orbit exactly once. For $i\in J$, note that $\sigma_g(\P^i)=\P^i$, for every $g\in\text{Stab}_G(i)$. Thus, we have a trace preserving action $\text{Stab}_G(i)\curvearrowright \P^i\cong \P$. We denote by $G\curvearrowright^{\sigma_i}\P^{G/\text{Stab}_G(i)}$ the co-induced action. Then, as explained right after \cite[Definition 2.5]{KV15}, $\sigma$ is conjugate to the product of co-induced actions  $\otimes_{i\in J}\sigma_i$.
\end{enumerate}
\end{ex}

The observation from \cite{KV15} recalled in Example \ref{exmpl} (c) implies the following.

 \begin{lem} \label{coind}
Let $A,B$ be countable groups and $B\curvearrowright I$ an action on a countable set $I$. Let $G\in\mathcal W\mathcal R(A,B\curvearrowright I)$, $\varepsilon\colon G\rightarrow B$ the quotient homomorphism  and  $(u_g)_{g\in G}$ the canonical generating unitaries of $\emph{L}(G)$.  Let  $G\curvearrowright I$ and $G\curvearrowright^{\sigma} \emph{L}(A^{(I)})=\emph{L}(A)^I$ be  the action  and  the trace preserving action given by $g\cdot i=\varepsilon(g) i$ and $\sigma_g=\emph{Ad}(u_g)$, for every $g\in G$ and $i\in I$.

Then $\sigma$ is built over $G\curvearrowright I$. Moreover, let $J\subset I$ be a set which meets each $G$-orbit exactly once. For $i\in J$, consider the trace preserving action  $\emph{Stab}_G(i)\curvearrowright^{\rho_i} \emph{L}(A_i)$ given by $(\rho_i)_g(u_h)=u_{ghg^{-1}}$, for every $g\in\emph{Stab}_G(i)$ and $h\in A_i$.
 Let $\sigma_i$ be the action of $G$ obtained by co-inducing $\rho_i$. Then $\sigma$ is conjugate to $\otimes_{i\in J}\sigma_i$.
 \end{lem}

\begin{proof}
Since $\sigma_g(\text{L}(A_i))=\text{L}(A_{g\cdot i})$, for every $g\in G$ and $i\in I$, $\sigma$  is built over $G\curvearrowright I$.
If $i\in J$ and $g\in\text{Stab}_G(i)$, the restriction of $\sigma_g$ to $\text{L}(A_i)$ is $(\rho_i)_g$, and the conclusion follows. \end{proof}

For further reference, we record the following consequence of Lemma \ref{coind}.

\begin{rem}\label{bernoulli}
Assume that $A$ is abelian. Then the conjugation action of $G$ on $A^{(I)}$ gives rise to an action $B=G/A^{(I)}\curvearrowright^{\alpha}\text{L}(A^{(I)})$. Explicitly, for $g\in B$, we have
$\alpha_g=\sigma_{\widehat{g}}$, where $\widehat{g}\in G$ is any element such that $\varepsilon(\widehat{g})=g$. Lemma \ref{coind} implies that $\alpha$ is built over $B\curvearrowright I$. Moreover, $\alpha$ is conjugate to $\otimes_{i\in J}\alpha_i$, where $\alpha_i$ is obtained by co-inducing the action $\text{Stab}_B(i)\curvearrowright^{\tau_i}\text{L}(A_i)$ given by $(\tau_i)_g=(\rho_i)_{\widehat{g}}$, for every $g\in\text{Stab}_B(i)$.
In particular, if $I=B$ endowed with the left multiplication action of $B$, then  $\sigma$ and $\alpha$ are conjugate to the generalized Bernoulli actions $G\curvearrowright \text{L}(A)^B$ and $B\curvearrowright\text{L}(A)^B$, respectively.
\end{rem}

  In the proof of Theorem \ref{superr}, we will use Lemma \ref{coind} in combination with the following extension of Popa's cocycle superrigidity theorems.


\begin{thm}\label{builtover} Let $G$ be a countable group with property (T), $G\curvearrowright I$ be an action on a countable set $I$ with infinite orbits and $(\P,\tau)$ be a tracial von Neumann algebra. Suppose that $G\curvearrowright^{\sigma}(\P^I,\tau)$ is a trace preserving action built over $G\curvearrowright I$.
Then the following hold.
\begin{enumerate}
\item[(a)] Any $1$-cocycle for $\sigma$ is cohomologous to a character of $G$. More generally, given a trace preserving action $G\curvearrowright^{\lambda}(\Q,\tau)$, any $1$-cocycle $w\colon G\rightarrow\sU(\P^I\,\overline{\otimes}\,\Q)$ for the product action $\sigma\otimes\lambda$ is cohomologous to a $1$-cocycle taking values into $\sU(\Q)\subset \sU(\P^I\,\overline{\otimes}\,\Q)$.
\item[(b)] Any generalized $1$-cocycle for $\sigma$ has support projection $1$.
\end{enumerate}
\end{thm}

Theorem \ref{builtover} extends results of Popa in \cite{Po01a,Po05} which cover Connes-St{\o}rmer and classical Bernoulli actions.
If $G\curvearrowright I$ has finitely many orbits, part (a) is a consequence of \cite[Theorem 3.1]{Dr15}. In general, Theorem \ref{builtover} follows by adapting the proof of \cite[Theorem 7.1]{VV14}.
We explain this briefly below, leaving the details to the reader.

 \begin{proof}
 Assume first that the action $G\curvearrowright I$ has finitely many orbits.  By \cite{KV15} (see Example \ref{exmpl} (c)), $\sigma$ is a product of finitely many co-induced actions. Since $G$ has property (T) and $G\curvearrowright I$ has infinite orbits, part (a) follows from \cite[Theorem 3.1]{Dr15}.

In general, adapting the proof of \cite[Theorem 7.1]{VV14} shows that Theorem \ref{builtover} holds if $\sigma$ is the generalized Bernoulli action $G\curvearrowright(\P^I,\tau)$ associated to the action $G\curvearrowright I$.
 To see this, assume the notation from \cite[Theorem 7.1]{VV14}.
Since $G$ has property (T), Step 2 in the proof of \cite[Theorem 7.1]{VV14} holds for $\Sigma=G$. Then Steps 3-6 in that proof, which only use Step 2 and that $G\curvearrowright I$ has infinite orbits, also hold for $\Sigma=G$. This justifies our claim.
Similarly to \cite[Theorem 2.6]{KV15}, the above proof can be reproduced verbatim to get the conclusion under the more general assumption that $\sigma$ is built over $G\curvearrowright I$.
 \end{proof}

\subsection{Cartan subalgebras and equivalence relations}

In this subsection, we first recall the connection between Cartan subalgebras and countable equivalence relations, and then record two conjugacy results for Cartan subalgebras. 

 If $G\curvearrowright (X,\mu)$ is a p.m.p. action of a countable group $G$, then its {\it orbit equivalence (OE) relation} $\sR(G\curvearrowright X)=\{(x_1,x_2)\in X^2\mid G\cdot x_1=G\cdot x_2\}$ is countable p.m.p. Conversely, every countable p.m.p. equivalence relation $\sR$ on $(X,\mu)$ arises this way \cite[Theorem 1]{FM77a}.
The {\it full group} of  $\sR$, denoted by $[\sR]$, consists of  all automorphisms $\theta$ of $(X,\mu)$ such that $(\theta(x),x)\in\sR$, for almost every $x\in X$. For a 2-cocycle $c\in\text{Z}^2(\sR,\mathbb T)$, we denote by $\text{L}_c(\sR)$ the tracial von Neumann algebra associated to $\sR$ and $c$ \cite[Section 2]{FM77b}. It is generated by a copy of $\text{L}^{\infty}(X)$ and unitaries $(u_{\theta})_{\theta\in [\sR]}$ such that $u_{\theta}au_{\theta}^*=a\circ\theta^{-1}$, for every $a\in \text{L}^{\infty}(X)$ and $\theta\in [\sR]$. When $c\equiv 1$ is the trivial $2$-cocycle, we use the notation $\text{L}(\sR)$.

Let $\M$ be a II$_1$ factor and $\A\subset \M$ be a Cartan subalgebra. Identify $\A=\text{L}^{\infty}(X)$, for a standard probability space $(X,\mu)$. For $u\in\sN_\M(\A)$, let $\theta_u$ be a measure space automorphism of $(X,\mu)$ such that $uau^*=a\circ\theta_u^{-1}$, for every $a\in A$. 
The equivalence relation of the inclusion $\A\subset \M$, denoted $\sR:=\sR(\A\subset \M)$, is the smallest countable p.m.p. equivalence relation on $(X,\mu)$ such that $\theta_u\in [\sR]$, for every $u\in\sN_\M(\A)$.
Then there is a $2$-cocycle $c\in\text{Z}^2(\sR,\mathbb T)$ such that the inclusion $(\A\subset \M)$ is isomorphic to $(\text{L}^{\infty}(X)\subset \text{L}_c(\sR))$ \cite[Theorem 1]{FM77b}.

The following two lemmas are extracted from the proofs of Theorems 6.1 and 8.2 in \cite{Io10}, respectively. However, for the reader's convenience, we include detailed proofs.

\begin{lem}[\cite{Io10}]\label{conj1}
Let $\M$ be a II$_1$ factor, $\A\subset \M$ be a Cartan subalgebra and $\D\subset \M$ be an abelian von Neumann subalgebra. Let $\C=\D'\cap \M$ and assume that $\C\prec_{\M}^{s}\A$. Then there exists $u\in\sU(\M)$ such that $\D\subset u\A u^*\subset \C$.
\end{lem}

\begin{proof}
Let $\C_0\subset \C$ be a maximal abelian von Neumann subalgebra. Then $\C_0$ contains $\sZ(\C)$ and hence $\D$. Thus, $\C_0'\cap \M\subset \D'\cap \M=\C$, and hence $\C_0$ is maximal abelian in $\M$.

Let $p\in \C_0$ be a nonzero projection.
Since $\C\prec_{\M}^{\text s}\A$, we get that $\C_0p\prec_{\M}\A$. Since $\C_0,\A\subset \M$ are maximal abelian, \cite[Theorem A.1]{Po01b} (see also \cite[Lemma C.3]{Va06})
provides nonzero projections $p'\in \C_0p$ and $q\in \A$ such that  $\C_0p'=v(\A q)v^*$, for a partial isometry $v\in \M$ satisfying $vv^*=p'$ and $v^*v=q$. Moreover, since $\A\subset \M$ is a Cartan subalgebra and $\M$ is a II$_1$ factor, the same holds if $q$ is replaced by any projection $q'\in \A$ with $\tau(q')=\tau(q)$.

By using this fact and a maximality argument, we can find projections $(p_i)_{i\in I}\subset \C_0$, $(q_i)_{i\in I}\subset \A$ and partial isometries $(v_i)_{i\in I}\subset \M$ such that we have $\sum_{i\in I}p_i=\sum_{i=I}q_i=1$ and $C_0p_i=v_i(\A q_i)v_i^*, v_iv_i^*=p_i,v_i^*v_i=q_i$, for every $i\in I$. It follows that $u=\sum_{i\in I}v_i$ is a unitary in $\M$ such that $\C_0=u\A u^*$. Thus, $\D\subset u\A u^*\subset \C$, which proves the conclusion. \end{proof}

\begin{lem}[\cite{Io10}]\label{conj2}
Let $\M$ be a II$_1$ factor, $\A\subset \M$ a Cartan subalgebra, $\D\subset \M$ an abelian von Neumann subalgebra and let $\C=\D'\cap \M$. Assume that $\C\prec_{\M}^{\text{s}}\A$ and $\D\subset \A\subset \C$.
Let $(\alpha_g)_{g\in G}$ be an action of a group $G$ on $\C$ such that $\alpha_g=\emph{Ad}(u_g)$, for some $u_g\in\sN_\M(\D)$, for every $g\in G$.
 Assume that the restriction of the action $(\alpha_g)_{g\in G}$  to $\D$ is free.

Then there is an action $(\beta_g)_{g\in G}$ of $G$ on $\C$ such that
\begin{enumerate}[label={(\alph*)}, ref=(\alph*)]
\item  \label{unu} for every $g\in G$ we have that $\beta_g=\alpha_g\circ\emph{Ad}(\omega_g)=\emph{Ad}(u_g\omega_g)$, for some $\omega_g\in\sU(\C)$, and
\item \label{doi}$\A$ is $(\beta_g)_{g\in G}$-invariant and the restriction of
 $(\beta_g)_{g\in G}$ to $\A$ is free.
\end{enumerate}
Moreover, if the action $(\alpha_g)_{g\in G}$ on $\C$ is weakly mixing, then we can find
 $(\beta_g)_{g\in G}$-invariant projections $p_1,...,p_k\in\A$ with
 $\sum_{j=1}^kp_i=1$, for some $k\in\mathbb N$, such that the restriction of $(\beta_g)_{g\in G}$ to $\A p_j$ is weakly mixing, for every $1\leq j\leq k$.

\end{lem}

\begin{rem} Assume the notation of Lemma \ref{conj2}.
If ${(u_g)}_{g\in G}\subset\sN_\M(\D)$ are such that  $u_{gh}^*u_gu_h\in\D$, for every $g,h\in G$, then $(\text{Ad}(u_g))_{g\in G}$ defines an action of $G$ on $\C$.
\end{rem}

For $k\in\mathbb N$, we denote by $\mathbb D_k(\mathbb C)\subset\mathbb M_k(\mathbb C)$ the subalgebra of diagonal matrices.

\begin{proof}
Since $\C\prec_{\M}^{\text{s}}\A$, $\C$ is a type I algebra. We can thus decompose $\sZ(\C)=\bigoplus_{i\geq 1}\sZ_i$ such that $\C=\bigoplus_{i\geq 1}(\sZ_i\,\overline{\otimes}\,\mathbb M_{k_i}(\mathbb C))$ for a strictly increasing, possibly finite, sequence $(k_i)\subset\mathbb N$.
Since
any two maximal abelian subalgebras of a type I algebra are unitarily conjugate (see, e.g., \cite[Lemma C.2]{Va06}) we may assume that $\A=\bigoplus_{i\geq 1}(\sZ_i\,\overline{\otimes}\,\mathbb D_{k_i}(\mathbb C))$.

Since the action $(\alpha_g)_{g\in G}$  on $\C$ leaves $\sZ_i$ invariant, for every $i$, we can define a new action $(\beta_g)_{g\in G}$ on $\C$ by letting $$\text{$\beta_g=\bigoplus_{i\geq 1}({\alpha_g}_{|\sZ_i}\otimes\text{Id}_{\mathbb M_{k_i}(\mathbb C)})$, for every $g\in G$.}$$
If $g\in G$, then since the automorphisms $\alpha_g$ and $\beta_g$ of $\C$ are equal on its center, $\sZ$, by \cite[Corollary 9.3.5]{KR86} we can find $\omega_g\in \sU(\C)$ such that $\beta_g=\alpha_g\circ\text{Ad}(\omega_g)$,  which proves \ref{unu}.

To prove \ref{doi}, note first that $(\beta_g)_{g\in G}$ leaves $\A$ invariant.
Second, let $g\in G$  such that $\beta_g(x)=x$, for every $x\in \A p$, for some nonzero projection $p\in \A$. We may assume that $p=z\otimes q$, where $z\in \sZ_i$ is a nonzero projection and $q\in\mathbb D_{k_i}(\mathbb C)$ is a minimal projection, for some $i$.
Then $\alpha_g(x)=x$, for every $x\in\sZ_iz$. If $r\in \D$ denotes the support of $E_\D(z)$, then as $\D\subset \sZ$ and $\alpha_g(\D)=\D$ by projecting onto $\D$ we get that $\alpha_g(x)=x$, for every $x\in \D r$. Since $r\not=0$ and the restriction of $(\alpha_g)_{g\in\Gamma}$ to $\D$ is free, we get that $g=e$.
Thus, the restriction of $(\beta_g)_{g\in G}$ to $\A$ is free.

To prove the moreover assertion, assume that the action $(\alpha_g)_{g\in G}$ on $\C$ is weakly mixing. Then $\C$ is of type I$_k$, for some $k\in\mathbb N$, so we can write $\C=\sZ\,\overline{\otimes}\,\mathbb M_k(\mathbb C)$ and $\A=\sZ\,\overline{\otimes}\,\mathbb D_k(\mathbb C)$.
Let $q_1,\cdots, q_k$ be the minimal projections of $\mathbb D_k(\mathbb C)$.
Then ${p_j=1\otimes q_j}\in \A$ is $(\beta_g)_{g\in G}$-invariant, for every $1\leq j\leq k$. Since the restriction of $(\alpha_g)_{g\in G}$ to $\sZ$ is weakly mixing, so is the restriction of $(\beta_g)_{g\in G}$ to $\A p_j=\sZ\otimes q_j$. This finishes the proof.
\end{proof}

\subsection{An intertwining result for property (T) subalgebras}
We end this section by using Popa and Vaes' structure theorem for normalizers in crossed products arising from actions of hyperbolic {groups} \cite{PV12}  to establish the following result.

\begin{thm}\label{relativeT}
Let $G,H$ be countable groups and $\delta\colon G\rightarrow H$ a homomorphism, where $H$ is hyperbolic. Let $G\curvearrowright (\Q,\tau)$ be a trace preserving action on a tracial von Neumann algebra  $(\Q,\tau)$ and $\M=\Q\rtimes G$.
 Let $\P\subset p\M p$ be a von Neumann subalgebra which is amenable relative to $\Q\rtimes\ker(\delta)$. 
 Let $\Nn=\sN_{p\M p}(\P)''$ and assume there is a von Neumann subalgebra $\R\subset \Nn$ with the relative property (T) such that $\R\nprec_{\M}\Q\rtimes\ker(\delta)$.
 Then  $\P\prec_{\M}^{\text s}\Q\rtimes\ker(\delta)$.
 \end{thm}

\begin{proof}  Let $(u_g)_{g\in G}\subset\sU(\M)$ and $(v_h)_{h\in H}\subset\sU(\text{L}(H))$ be the canonical unitaries. Following \cite[Section 3]{CIK13}, define $\Delta\colon \M\rightarrow \M\,\overline{\otimes}\,\text{L}(H)$ by letting $\Delta(xu_g)=xu_g\otimes v_{\delta(g)}$, for every $x\in \Q$ and $g\in G$.
Write $\M\,\overline{\otimes}\,\text{L}(H)=\M\rtimes H$, where $H$ acts trivially on $\M$.
Before proving the conclusion, we recall the following fact proved in \cite[Proposition 3.4]{CIK13}.

\begin{fact}\label{CIK}
If $\Ss\subset q\M q$ is a von Neumann subalgebra such that $\Delta(\Ss)\prec_{\M\,\overline{\otimes}\,\text{L}(H)}\M\,\overline{\otimes}\,\text{L}(\Sigma)$, for some subgroup $\Sigma<H$, then $\Ss\prec_\M\Q\rtimes\delta^{-1}(\Sigma)$.
\end{fact}

Assume by contradiction that the conclusion is false. Then  \cite[Lemma 2.4(2)]{DHI16} provides a nonzero projection $z\in \Nn'\cap p\M p$ such that $\P z\nprec_{\M}\Q\rtimes\ker(\delta)$.
 {Since $\P$ is amenable relative to $\Q\rtimes\text{ker}(\delta)$ we get that $\Delta(\P)$ is amenable relative to $\Delta(\Q\rtimes\ker(\delta))=\Q\,\overline{\otimes}\, 1$ and thus to $\M\,\overline{\otimes}\,1$}.
Since $H$ is  hyperbolic, applying \cite[Theorem 1.4]{PV12} to $\Delta(\P z)\subset \M\,\overline{\otimes}\,\text{L}(H)$ gives that either 1) $\Delta(\P z)\prec_{\M\,\overline{\otimes}\,\text{L}(H)}\M\,\overline{\otimes}\, 1$ or 2) $\Delta(\Nn z)$ is amenable relative to $\M\,\overline{\otimes}\, 1$ inside $\M\,\overline{\otimes}\,\text{L}(H)$.

If 1) holds, then Fact \ref{CIK} gives that $\P z\prec_{M}\Q\rtimes\ker(\delta)$, which is a contradiction.
If 2) holds, then there is a sequence $\eta_n\in \text{L}^2(\Delta(z)\langle \M\,\overline{\otimes}\,\text{L}({H}),\M\,\overline{\otimes}\,1\rangle\Delta(z))^{\bigoplus\infty}$ such that $\|\langle \cdot\eta_n,\eta_n\rangle-\tau(\cdot)\|\rightarrow 0$, $\|\langle\eta_n\cdot,\eta_n\rangle-\tau(\cdot)\|\rightarrow 0$, and $\|y\eta_n-\eta_n y\|_2\rightarrow 0$, for every $y\in\Delta(\Nn z)$ (see Remark \ref{TT}).
 Since  $\R\subset \Nn$ has the relative property (T),   by \cite[Proposition 4.7]{Po01b},  so does $\Delta(\R z)\subset\Delta(\Nn z)$. Hence, there is a nonzero $\eta\in \text{L}^2(\Delta(z)\langle \M\,\overline{\otimes}\,\text{L}(H),\M\,\overline{\otimes}\, 1\rangle\Delta(z))$ such that $y\eta=\eta y$, for every $y\in\Delta(\R z)$.
Then $\zeta=\eta^*\eta\in  \text{L}^1(\Delta(z)\langle \M\,\overline{\otimes}\,\text{L}(H),\M\,\overline{\otimes}\, 1\rangle\Delta(z))$ is nonzero and satisfies $\zeta\geq 0$ and $y\zeta=\zeta y$, for every $y\in\Delta(\R z)$. Let $t>0$ such the spectral projection $a={\bf 1}_{[t,\infty)}(\zeta)$ of $\zeta$ is nonzero. Then $a\in \Delta(\R z)'\cap \Delta(z)\langle \M\,\overline{\otimes}\,\text{L}(H),\M\,\overline{\otimes}\, 1\rangle\Delta(z)$. As $ta\leq\zeta$, we get that $\text{Tr}(a)\leq{\text{Tr}(\zeta)}/{t}<\infty$.  Theorem \ref{corner} implies that $\Delta(\R z)\prec_{\M\,\overline{\otimes}\,\text{L}(H)}\M\,\overline{\otimes}\, 1$.  Applying Fact \ref{CIK} again, we get that $\R\prec_{\M}\Q\rtimes\ker (\delta)$, a contradiction.
\end{proof}


\section{$W^*$-superrigid groups with property (T)}\label{BER}


The goal of this section is to prove Theorem \ref{superr}. 
 We begin with an informal outline the proof of Theorem \ref{superr}.   For simplicity, let $G\in\mathcal W\mathcal R(A,B)$ be a property (T) group, where $A$ is nontrivial abelian and $B$ is an ICC subgroup of a hyperbolic group. Denote $\mathcal M=\text{L}(G)$ and assume that  $\mathcal M=\text{L}(H)$, for some arbitrary group $H$. We denote by $(u_g)_{g\in G}\subset \text{L}(G)$ and $(v_h)_{h\in H}\subset \text{L}(H)$  the canonical generating unitaries.

The proof of Theorem \ref{superr} is based on a deformation/rigidity strategy, which plays property (T) against two key properties of wreath-like product groups that relate them to wreath product groups. Namely, letting $\mathcal P=\text{L}(A^{(B)})$, we have:

\begin{enumerate} 
\item[(i)]\label{i}
The action $G\curvearrowright^{\sigma}\mathcal P=\text{L}(A)^B$ given by $\sigma_g=\text{Ad}(u_g)$ is a generalized Bernoulli action. 
\item[(ii)]\label{ii}
 $\mathcal P\subset\mathcal M$ is a Cartan subalgebra and  $\mathcal R(\mathcal P\subset\mathcal M)$ is the orbit equivalence
(OE) relation of the Bernoulli action $B\curvearrowright\widehat{A}^B$, where $\widehat{A}$ is the dual of $A$.
\end{enumerate}

Specifically, rather than (ii), we use the following ``transfer principle" implied by (ii).  Let $\mathcal N=\text{L}(A\wr B)$. 
If $\mathcal P\,\overline{\otimes}\,\mathcal P\subset\mathcal D\subset\mathcal M\,\overline{\otimes}\,\mathcal M$ is a subalgebra, then $\mathcal R(\mathcal P\,\overline{\otimes}\,\mathcal P\subset\mathcal D)$ is a subequivalence {relation} of the OE relation of the product action $B\times B\curvearrowright \widehat{A}^B\times\widehat{A}^B$. So, there is subalgebra $\mathcal P\,\overline{\otimes}\,\mathcal P\subset\widetilde{\mathcal D}\subset\mathcal N\,\overline{\otimes}\,\mathcal N$ such that the inclusions $\mathcal P\,\overline{\otimes}\,\mathcal P\subset\mathcal D$ and $\mathcal P\,\overline{\otimes}\,\mathcal P\subset\widetilde{\mathcal D}$ have isomorphic equivalence relations. In particular, if $\widetilde{\mathcal D}$ is amenable, then $\mathcal D$ is amenable. The use of this transfer principle is a main novelty of our approach.

Define the comultiplication  $\Delta:\mathcal M\rightarrow\mathcal M\,\overline{\otimes}\,\mathcal M$ by letting $\Delta(v_h)=v_h\otimes v_h$, $h\in H$ \cite{PV09}.
 In the first part of the proof, following \cite{Io10,IPV10}, we analyze $\Delta$ and show that  $\mathcal D:=\Delta(\mathcal P)'\cap\mathcal M\,\overline{\otimes}\,\mathcal M$  is essentially unitarily conjugated to $\mathcal P\,\overline{\otimes}\,\mathcal P$.
Since $\Delta(\mathcal P)$ is amenable and has large normalizer, using Popa and Vaes' structure theorem \cite{PV12}  for normalizers inside crossed products by hyperbolic groups  as in \cite{CIK13} allows us to essentially show that $\Delta(\mathcal P)\subset\mathcal P\,\overline{\otimes}\,\mathcal P$, after unitary conjugacy. Next, as in \cite{BV13}, we use solidity results for generalized Bernoulli crossed products. Thus, applying the above transfer principle to $\mathcal D$ and extending the solidity theorem of \cite{CI08} (see Section \ref{SOL}), we derive that $\mathcal D$ is amenable. Another application of \cite{PV12} implies that $\mathcal D$ is essentially unitarily conjugated to $\mathcal P\,\overline{\otimes}\,\mathcal P$.

The second part of the proof is a ``discretization argument".
By the first part, we may assume that $\Delta(\mathcal P)'\cap\mathcal M\,\overline{\otimes}\,\mathcal  M=\mathcal P\,\overline{\otimes}\,\mathcal P$, after unitary conjugacy.
In particular, the group $\Delta(G)=(\Delta(u_g))_{g\in G}$ normalizes $\mathcal P\,\overline{\otimes}\,\mathcal P$. Moreover, the resulting action of $\Delta(G)$ on $\mathcal P\,\overline{\otimes}\,\mathcal P$ descends to a free action of $\Delta(B)$.  A second application of our transfer principle gives a free action of $\Delta(B)\cong B$ on $\widehat{A}^B\times\widehat{A}^B$ whose OE relation is contained in that of $B\times B\curvearrowright \widehat{A}^B\times\widehat{A}^B$. Since $B$ has property (T), a generalization of a theorem from \cite{Po04} (see Section \ref{STRONG}) allows us to assume that $\Delta(B)\subset B\times B$, as groups of automorphisms of $\widehat{A}^B\times\widehat{A}^B$.
Consequently, $\Delta(G)$ ``discretizes" modulo $\mathscr U(\mathcal P\,\overline{\otimes}\,\mathcal P)$: there are maps $\delta_1,\delta_2:G\rightarrow G$  and $\omega:G\rightarrow\mathscr U(\mathcal P\,\overline{\otimes}\,\mathcal P)$ such that $\Delta(u_g)=\omega_g(u_{\delta_1(g)}\otimes u_{\delta_2(g)})$, for every $g\in G$.

In the last part of the proof, we first use the symmetry and associativity properties of $\Delta$ to show that we may take $\delta_1$ and $\delta_2$ to be the identity of $G$. In other words,  we have $\Delta(u_g)=\omega_g(u_g\otimes u_g)$, for every $g\in G$.  So far we have only used that $B$, but not $G$, has property (T).
Another main novelty of this paper is the way we use property (T) for $G$. 
 We start by observing that as $G$ has property (T) and $\sigma$ is a generalized Bernoulli action by (i), Popa's cocycle superrgidity theorem \cite{Po05} implies that any $1$-cocycle for $\sigma\otimes\sigma$ is cohomologous to a character of $G$.
Thus, since $(\omega_g)_{g\in G}$ is a $1$-cocycle for $\sigma\otimes\sigma$, we can  find a unitary $w\in \mathcal P\,\overline{\otimes}\,\mathcal P$ and a character $\rho:G\rightarrow \mathbb T$ such that $w\Delta(u_g)w^*=\rho(g)(u_g\otimes u_g)$, for every $g\in G$. But then a general result from \cite{IPV10} implies the conclusion of Theorem \ref{superr}.


\subsection{Strong rigidity for orbit equivalence embeddings}\label{STRONG}


Popa's deformation rigidity/theory has been used to derive a number of powerful rigidity results for von Neumann algebras associated to Bernoulli actions.
To prove Theorem \ref{superr}, we need to extend two of such results from plain to generalized Bernoulli actions.

Let $B\curvearrowright (X,\mu)=(Y^{B},\nu^{B})$ be a Bernoulli action of a countable group $B$.
In \cite{Po03}, Popa discovered his malleable deformation of  the crossed product $\M=\text{L}^{\infty}(X)\rtimes B$. He used this in \cite{Po03,Po04} to prove a series of rigidity results  under property (T) assumptions.  In particular, in \cite[Theorem 0.5]{Po04}, he obtained the following strong rigidity
theorem for orbit equivalence embeddings: if $B$ is ICC and $H\curvearrowright (X,\mu)$ is a free ergodic p.m.p.\ action of an ICC group $H$ admitting an infinite normal subgroup with the relative property (T) such that $H\cdot x\subset B\cdot x$, for almost every $x\in X$, then $\theta\circ H\circ\theta^{-1}\subset B$,  for some $\theta\in [\sR(B\curvearrowright X)]$.

In \cite{Po06b}, Popa introduced his spectral
gap rigidity principle and combined it with the deformation/rigidity methods of \cite{Po03,Po04} to prove solidity results for $M$. These methods were combined with those of
 \cite{Io06} in \cite{CI08} to prove the following relative solidity theorem:  $\Q'\cap \M$ is amenable, for any diffuse subalgebra $\Q\subset \text{L}^{\infty}(X)$.

In the proof of Theorem \ref{superr}, we will need analogues of the above {\it strong rigidity for OE embeddings} and  {\it relative solidity} results for the product action $ B\times B\curvearrowright (X\times X,\mu\times\mu)$. To this end, we use results from \cite{IPV10} to extend these results to certain classes of generalized Bernoulli actions  which include the action $B\times B\curvearrowright (X\times X,\mu\times\mu)$.
More generally, we treat measure preserving actions $C\curvearrowright (Z^J,\rho^J)$ which are built over an action $C\curvearrowright J$, i.e., such that the associated trace preserving action $C\curvearrowright\text{L}^{\infty}(Z)^J$ is built over $C\curvearrowright J$ in the sense of \cite[Definition 2.5]{KV15} (see Example \ref{exmpl} (c)).

First, in this section, we extend \cite[Theorem 0.5]{Po04} to a large class of generalized Bernoulli actions. Although the next statement is ergodic-theoretic, as in \cite{Po04}, its proof relies crucially on the framework of von Neumann algebras.

\begin{thm}\label{SOE}
Let $ B$ be an ICC group and $ B\curvearrowright^{\alpha}(X,\mu)=(Y^I,\nu^I)$ be a measure preserving action built over an action $B\curvearrowright I$, where $(Y,\nu)$ is a probability space.  Let $\widetilde B= B\times\mathbb Z/n\mathbb Z$ and $(\widetilde X,\widetilde\mu)=(X\times\mathbb Z/n\mathbb Z,\mu\times c)$, where $n\in\mathbb N$ and $c$ is the counting measure of $\mathbb Z/n\mathbb Z$. Consider the action $\widetilde B\curvearrowright^{\widetilde\alpha} (\widetilde X,\widetilde\mu)$ given by $(g,a)\cdot (x,b)=(g\cdot x,a+b)$.

Let $D$ be a countable group with a normal subgroup $D_0$ such that the pair $(D,D_0)$ has the relative property (T).
Let $X_0\subset \widetilde X$ be a measurable, {non-negligible} set  and $D\curvearrowright^{\beta} (X_0,\widetilde\mu_{|X_0})$ be a weakly mixing free measure preserving action such that  $D\cdot x\subset\widetilde B\cdot x$, for almost every $x\in X_0$.
Assume that for every $i\in I$, there is a sequence $(h_m)\subset D_0$ such that  for every  $s,t\in\widetilde B$ we have $\widetilde\mu(\{x\in X_0\mid h_m\cdot x\in s(\emph{Stab}_B(i)\times\mathbb Z/n\mathbb Z)t\cdot x\})\rightarrow 0$, as $m\rightarrow\infty$.

Then there exist a subgroup $B_1< B$, a finite normal subgroup $K\lhd B_1$, an isomorphism $\delta\colon D\rightarrow B_1/K$, a  measurable set $X_1\subset X\equiv X\times\{0\}$ and $\theta\in [\sR(\widetilde B\curvearrowright\widetilde X)]$ such that
\begin{enumerate}
\item[(a)] $X_1$ is a fundamental domain for $\alpha_{|K}$, i.e., $X=\bigsqcup_{k\in K}\alpha(k)(X_1)$,
\item[(b)] $\theta(X_0)=X_1$, so in particular $\widetilde\mu(X_0)=\mu(X_1)=|K|^{-1}\leq 1$, and
\item[(c)] $\theta\circ\beta(h)=\gamma(\delta(h))\circ\theta$, for every $h\in D$, where $B_1/K\curvearrowright^{\gamma}(X_1,\mu_{|X_1})$ is the action given by $\{\gamma(gK)x\}=\alpha(gK)x\cap X_1$, for every $g\in B_1$ and $x\in X_1$.
\end{enumerate}

Assume additionally that for every $g\in B\setminus\{1\}$, there is a sequence $(l_m)\subset D$ such that $\widetilde\mu(\{x\in X_0\mid l_m\cdot x\in s(\emph{C}_B(g)\times\mathbb Z/n\mathbb Z)t\cdot x\})\rightarrow 0$, as $m\rightarrow\infty$, for all $s,t\in\widetilde B$. Then $K=\{1\}$.
\end{thm}

\begin{rem}
The action $B_1/K\curvearrowright^{\gamma}(X_1,\mu_{|X_1})$ is isomorphic to the natural quotient action $B_1/K\curvearrowright (X/K,\bar{\mu})$, where $\bar{\mu}$ is the push-forward of $\mu$ through the quotient map $X\rightarrow X/K$.
\end{rem}

The proof of Theorem \ref{SOE} relies on the following result {which is a direct consequence of} \cite{IPV10}.
\begin{cor} [\cite{IPV10}]\label{rigsub}
Let $ B$ be an ICC group, $B\curvearrowright I$ be an action, $(\A,\tau)$ be a tracial von Neumann algebra and $B\curvearrowright (\A^I,\tau)$ be a trace preserving action built over $B\curvearrowright I$.
Let $\M=\A^I\rtimes B$ and $(\Nn,\tau)$ be a tracial factor. Let $\Q\subset p(\M\,\overline{\otimes}\,\Nn)p$ be a von Neumann subalgebra with the relative property (T) such that $\Q\nprec_{\M\,\overline{\otimes}\,\Nn}(\A^I\rtimes\emph{Stab}_B(i))\,\overline{\otimes}\,\Nn$, for all $i\in I$. Let $\P=\mathcal N_{p(\M\,\overline{\otimes}\,\Nn)p}(\Q)''$.

Then there exists $v\in\M\,\overline{\otimes}\,\Nn$ such that $v^*v=p$ and $v\P v^*\subset\emph{L}(B)\,\overline{\otimes}\,\Nn$.
\end{cor}

\begin{proof}
Assume that $B\curvearrowright \A^I$ is the generalized Bernoulli action associated to $B\curvearrowright I$.
For $\mathcal F\subset I$,  let $\sM_{\mathcal F}=(\A^{\mathcal F}\rtimes\text{Stab}_B(\mathcal F))\,\overline{\otimes}\,\Nn$. Let $\sM_{\emptyset}=\text{L}(B)\,\overline{\otimes}\,\Nn$.
Let $p_0\in\mathcal \sZ(\P)$ be a nonzero projection. Since $\Q\nprec_{\M\,\overline{\otimes}\,\Nn}(\A^I\rtimes\text{Stab}_B(i))\,\overline{\otimes}\,\Nn$, for all $i\in I$, and  $\Q p_0\subset p_0(\M\,\overline{\otimes}\,\Nn)p_0$ has the relative property (T), the proof of \cite[Theorem 4.2]{IPV10} shows that $\Q p_0\prec_{\M\,\overline{\otimes}\,\Nn}\sM_{\mathcal F}$, for a finite, possibly empty, set $\mathcal F\subset I$.
 {We claim that $\mathcal F=\emptyset$. Otherwise, if $i\in \mathcal F$, then $\text{Stab}_B(\mathcal F)\subset\text{Stab}_B(i)$ and thus $\sM_{\mathcal F}\subset (\A^I\rtimes\text{Stab}_B(i))\,\overline{\otimes}\,\Nn$, which would imply that $\Q p_0\prec_{\M\,\overline{\otimes}\,\Nn}(\A^I\rtimes\text{Stab}_B(i))\,\overline{\otimes}\,\Nn$, contradicting our assumption. Thus, $\mathcal F=\emptyset$ and therefore $\Q p_0\prec_{\M\,\overline{\otimes}\,\Nn}\mathcal M_{\emptyset}=\text{L}(B)\,\overline{\otimes}\,\Nn$.}
Further, \cite[Lemma 4.1(1)]{IPV10} implies that $\P p_0\prec_{\M\,\overline{\otimes}\,\Nn}\text{L}(B)\,\overline{\otimes}\,\Nn$. Since this holds for every nonzero projection $p_0\in\mathcal \sZ(\P)$ and $B$ is ICC, repeating the beginning of the proof of \cite[Corollary 4.3]{IPV10} gives the conclusion.

In general, when $B\curvearrowright \A^I$ is built over $B\curvearrowright I$,  the above proof and results from \cite[Section 4]{IPV10} carry over verbatim to give the conclusion  in this case.
\end{proof}

\begin{proof} [Proof of Theorem \ref{SOE}]Denote $\M=\text{L}^{\infty}(X)\rtimes B$, $\widetilde \M=\text{L}^{\infty}(\widetilde X)\rtimes\widetilde B$ and $\Nn=\text{L}^{\infty}(X_0)\rtimes D$.
Denote by $(u_g)_{g\in B}\subset \M$, $(\widetilde u_g)_{g\in\widetilde B}\subset\widetilde \M$ and $(v_h)_{h\in D}\subset \Nn$ the canonical unitaries.
For $h\in D$ and $g\in\widetilde B$, let $A_h^g=\{x\in X_0\mid h^{-1}\cdot x=g^{-1}\cdot x\}$.
Let $p_0={\bf 1}_{X_0}$ and consider the $*$-homomorphism $\pi\colon \Nn\rightarrow p_0\widetilde \M p_0$ given by $\pi(a)=a$ and $\pi(v_h)=\sum_{g\in\widetilde B}{\bf 1}_{A_h^g}\widetilde u_g$,  for every $a\in \text{L}^{\infty}(X_0)$ and $h\in D$. We view $\Nn$ as a subalgebra of $\widetilde \M$ by identifying it with $\pi(\Nn)$. We identify $\widetilde \M=\M\,\overline{\otimes}\,\mathbb M_n(\mathbb C)$ and endow $\widetilde\M$ with the normalized trace $\widetilde\tau=\tau\otimes n^{-1}\text{Tr}$.

 We first claim that for every $i\in I$ we have \begin{equation}
 \label{nonembed}\text{$\text{L}(D_0)\nprec_{\widetilde \M}(\text{L}^{\infty}(X)\rtimes \text{Stab}_{B}(i))\,\overline{\otimes}\,\mathbb M_n(\mathbb C)=\text{L}^{\infty}(\widetilde X)\rtimes (\text{Stab}_B(i)\times\mathbb Z/n\mathbb Z)$.}\end{equation}

Let $i\in I$. Put $B_0=\text{Stab}_{B}(i)$ and $\widetilde B_0= B_0\times\mathbb Z/n\mathbb Z$.  Let $s,t\in\widetilde B$ and $a,b\in \text{L}^{\infty}(\widetilde X)$ with $\|a\|,\|b\|\leq 1$.
If $h\in D_0$, then
 $E_{\text{L}^{\infty}(\widetilde X)\rtimes \widetilde B_0}(a\widetilde u_{s}v_{h}\widetilde u_{t}b)=a\big(\sum_{g\in s^{-1}\widetilde B_0t^{-1}}\widetilde\alpha(s)({\bf 1}_{A_h^g})\widetilde u_{sgt}\big)b$, thus $\|E_{\text{L}^{\infty}(\widetilde X)\rtimes \widetilde B_0}(a\widetilde u_{s}v_{h}\widetilde u_{t}b)\|_2^2\leq \sum_{g\in s^{-1}\widetilde B_0t^{-1}}\widetilde\mu(A_h^g)=\widetilde\mu(\{x\in\widetilde X\mid h^{-1}\cdot x\in t\widetilde B_0s\cdot x\})$
 Using this fact, the hypothesis gives a sequence $(h_m)\subset D_0$ such that $\|E_{\text{L}^{\infty}(\widetilde X)\rtimes \widetilde B_0}(a\widetilde u_{s}v_{h_m}\widetilde u_{t}b)\|_2\rightarrow 0$. Since this holds for every $s,t\in\widetilde B$ and $a,b\in \text{L}^{\infty}(\widetilde X)$ with $\|a\|,\|b\|\leq 1$, we conclude that $\|E_{\text{L}^{\infty}(\widetilde X)\rtimes \widetilde B_0}(cv_{h_m}d)\|_2\rightarrow 0$, for every $c,d\in\widetilde\M$,
 which proves \eqref{nonembed}.

Since $(D,D_0)$ has the relative property (T), so does the inclusion $\text{L}(D_0)\subset \text{L}(D)$ by \cite[Proposition 5.1]{Po01b}.
As $D_0\lhd D$ is normal and $ B$ is ICC, using \eqref{nonembed} and Corollary \ref{rigsub} we find a partial isometry $y\in \widetilde \M$ such that $yy^*=p_0$, $p\colon =y^*y\in \text{L}( B)\,\overline{\otimes}\,\mathbb M_n(\mathbb C)$, and \begin{equation}\label{conjugi}y^*\text{L}(D)y\subset p(\text{L}( B)\,\overline{\otimes}\,\mathbb M_n(\mathbb C))p. \end{equation}

Thus, the group of unitaries $(y^*v_hy)_{h\in D}\subset\sU(p(\text{L}( B)\,\overline{\otimes}\,\mathbb M_n(\mathbb C))p)$ normalizes $y^*\text{L}^{\infty}(X_0)y\subset p\widetilde \M p$. Moreover, the action $(\text{Ad}(y^*v_hy))_{h\in D}$ on $y^*\text{L}^{\infty}(X_0)y$ is isomorphic to $\beta$ and so is weakly mixing. By applying  \cite[Theorem 6.1]{IPV10}, whose conclusion holds with $d=1$ as $y^*\text{L}^{\infty}(X_0)y\subset p\widetilde \M p$ is a maximal abelian subalgebra, we get that there exist
\begin{itemize}
\item a subgroup $B_1< B$, a finite normal subgroup $K\lhd B_1$, a character $\rho\colon K\rightarrow\mathbb T$ such that
the associated projection $p_K=|K|^{-1}\sum_{k\in K}\rho(k)u_k$ commutes with $\{u_g\mid g\in B_1\}$,
\item an isomorphism $\delta\colon D\rightarrow B_1/K$,
\item an $\alpha(B_1)$-invariant projection $q\in \text{L}^{\infty}(X)$, and
\item a partial isometry $v\in \text{L}( B)\,\overline{\otimes}\,\mathbb M_{n,1}(\mathbb C)$  with $vv^*=p$
\end{itemize}
  such that
 $w=\tau(q)^{-1/2}vq$ is a partial isometry satisfying $ww^*=p$, $w^*w=p_Kq$, $w^*(y^*\text{L}^{\infty}(X_0)y)w=(\text{L}^{\infty}(X)q)^{K}p_K$ and we have that
\begin{equation}\label{Gamma1}\text{$w^*(y^*v_hy)w=\eta(h)u_{\widehat{\delta}(h)}p_Kq$, for every $h\in D$,}\end{equation}
where $\widehat{\delta}\colon D\rightarrow B_1$  and $\eta\colon D\rightarrow\mathbb T$ are maps such that $\widehat{\delta}(h)K=\delta(h)$, for every $h\in D$.

We next claim that $\alpha_{| B_1}$ is ergodic. Otherwise, we can find $i\in I$ such that $B_1\cap B_0<B_1$ has finite index, where $B_0=\text{Stab}_{ B}(i)$, for some $i\in I$ (see \cite[Proposition 2.3]{PV06}). By \eqref{nonembed} there is a sequence $(h_m)\subset D$ such that $\|E_{\text{L}( B_0)\,\overline{\otimes}\,\mathbb M_n(\mathbb C)}(av_{{h_m}}b)\|_2\rightarrow 0$, for all $a,b\in\widetilde \M$. Since $B_1\cap B_0<B_1$ has finite index,  we get  $\|E_{\text{L}(B_1)\,\overline{\otimes}\,\mathbb M_n(\mathbb C)}(av_{{h_m}}b)\|_2\rightarrow 0$, for all $a,b\in\widetilde \M$. On the other hand,  \eqref{Gamma1} implies that $\|E_{\text{L}(B_1)\,\overline{\otimes}\,\mathbb M_n(\mathbb C)}(w^*(y^*v_hy)w)\|_2=n^{-1/2}\|E_{\text{L}(B_1)}(p_Kq)\|_2\not=0$ for every $h\in D$. This gives a contradiction.

Since $\alpha_{|B_1}$ is ergodic, we further derive that $q=1$ and thus $w=v$. Let $z=yv$. Then $z$ is a partial isometry such that $zz^*=p_0,z^*z=p_K, z^*\text{L}^{\infty}(X_0)z=\text{L}^{\infty}(X)^{K}p_K$ and $z^*v_hz=\eta(h)u_{\widehat{\delta}(h)}p_K$, for every $h\in D$.

Let $X_1\subset X$ be a fundamental domain for $\alpha_{|K}$ and put $p_1={\bf 1}_{X_1}$. Then $t=|K|^{1/2}p_Kp_1$ is a partial isometry such that $tt^*=p_K,t^*t=p_1$ and $\text{L}^{\infty}(X)^Kp_K=t\text{L}^{\infty}(X_1)t^*$. Hence $\xi=zt$ is a partial isometry such that $\xi\xi^*=p_0$, $\xi^*\xi=p_1$ and $\xi^* \text{L}^{\infty}(X_0)\xi=\text{L}^{\infty}(X_1)$. Thus, we can find $\theta\in [\sR(\widetilde B\curvearrowright\widetilde X)]$ such that $\theta(X_0)=X_1$ and $\xi^*a\xi=a\circ\theta^{-1}$, for every $a\in \text{L}^{\infty}(X_0)$.
Moreover,  $t^*u_gp_Kt=\sum_{k\in K}\rho(k)p_1u_{gk}p_1$, for every $g\in B_1$. This implies that
\begin{equation}\label{conj}
\text{$\xi^*v_h\xi=\eta(h)\sum_{k\in K}\rho(k)p_1u_{\widehat{\delta}(h)k}p_1$, for every $h\in D$.}
\end{equation}
If $h\in D$, then $\text{Ad}(\xi^*v_h\xi)(a)=a\circ\theta\circ \beta(h)^{-1}\circ\theta^{-1}$, for every $a\in \text{L}^{\infty}(X_1)$. On the other hand, using \eqref{conj} it is easy to see that $\text{Ad}(\xi^*v_h\xi)(a)=a\circ\gamma(\delta(h))^{-1}$, for every $a\in \text{L}^{\infty}(X_1)$. We thus conclude that $\theta\circ\beta(h)=\gamma(\delta(h))\circ\theta$, for every $h\in D$, as claimed.

To prove the moreover assertion, assume that $K\not=\{1\}$ and let $g\in K\setminus\{1\}$. Let $(l_m)\subset D$ be a sequence such that $\widetilde\mu(\{x\in X_0\mid l_m\cdot x\in s(\emph{C}_B(g)\times\mathbb Z/n\mathbb Z)t\cdot x\})\rightarrow 0$, as $m\rightarrow\infty$, for all $s,t\in\widetilde B$. As in the proof of \eqref{nonembed} it follows that $\text{L}(D)\nprec_{\widetilde \M}\text{L}(\text{C}_{ B}(g))\,\overline{\otimes}\,\mathbb M_n(\mathbb C)$. Since the set $\{hgh^{-1}\mid h\in B_1\}\subset K$ is finite, $B_1\cap\text{C}_{ B}(g)<B_1$ has finite index, so  $\text{L}(D)\nprec_{\widetilde \M}\text{L}(B_1)\,\overline{\otimes}\,\mathbb M_n(\mathbb C)$. This contradicts the fact that $z^*\text{L}(D)z\subset\text{L}(B_1)p_K$.
\end{proof}

We end this section by showing that the weakly mixing condition from Theorem \ref{SOE} is automatically satisfied after passing to an ergodic component of a finite index subgroup.

\begin{lem}[\cite{Po04}]\label{wmix}
Assume the setting of Theorem \ref{SOE}.  Then there exists a finite index subgroup $ S< D$ and a $\beta(S)$-invariant non-null measurable set $Y\subset X_0$ such that $\widetilde\mu(\beta(h)(Y)\cap Y)=0$, for every $h\in D\setminus S$, and the restriction of $\beta_{| S}$ to $Y$ is weakly mixing.
\end{lem}

The proof follows from an argument Popa (see the proofs of \cite[Lemma 4.5]{Po04}, \cite[Theorem 9.1]{Va06} and \cite[Theorem 8.2]{Io10}).
For completeness, we reproduce the argument here.

\begin{proof}
Assume the notations from the proof of Theorem \ref{SOE}.
 {In particular, we recall that  $y\in\widetilde\M$ is a partial isometry such that $yy^*=p_0={\bf 1}_{X_0}$ and $y^*y=p$. Moreover, by \eqref{conjugi}, we have $y^*\text{L}(D)y\subset p(\text{L}( B)\,\overline{\otimes}\,\mathbb M_n(\mathbb C))p$.}
Let $\P_0\subset \text{L}^{\infty}(X_0)$ be the $*$-algebra of $f\in \text{L}^{\infty}(X_0)$ such that the linear span of $\{\beta(h)(f)\mid f\in D\}$ is finite dimensional. Let $\P\subset \text{L}^{\infty}(X_0)$ be the von Neumann algebra generated by $\P_0$.

Let $f\in \P_0$ and denote by $\mathcal H$ the linear span of $\{\beta(h)(f)\,|\, {h\in D}\}$. Then, $v_hf=\beta(h)(f)v_h\in\mathcal Hv_h$,  for every $h\in D$. This implies that $\text{L}(D)f\subset\mathcal H\text{L}(D)$ and thus $(y^*\text{L}(D)y)(y^*fy)\subset (y^*\mathcal Hy)(y^*\text{L}(D)y)$. Since $\mathcal H$ is finite dimensional, by using \eqref{nonembed} and \eqref{conjugi} and applying \cite[Propositions 6.14 and 6.15]{PV06} we get that $y^*fy\in p(\text{L}(B)\,\overline{\otimes}\,\mathbb M_n(\mathbb C))p$. Thus, we get that $y^*\P y\subset p(\text{L}(B)\,\overline{\otimes}\,\mathbb M_n(\mathbb C))p$ and so $\P\subset \text{L}^{\infty}(X_0)\cap y(L(B)\,\overline{\otimes}\,\mathbb M_n(\mathbb C))y^*$.

This easily implies that $\P$ is completely atomic. Let $Y\subset X_0$ be a non-null measurable set such that ${\bf 1}_{Y}$ is a minimal projection of $P$. Let $ S<D$ be the subgroup of $h\in D$ such that $\beta(h)(Y)=Y$. Then $ S$ has finite index in $D$ and
$\widetilde\mu(\beta(h)(Y)\cap Y)=0$, for every $h\in D\setminus S$. If $f\in \text{L}^{\infty}(Y)$ is such that the linear span of $\{\beta(h)(f)\mid h\in S\}$ is finite dimensional, then $f\in \P$ and hence $f\in\mathbb C{\bf 1}_Y$. This shows that the restriction of $\beta_{| S}$ to $Y$ is weakly mixing.
\end{proof}


\subsection{Solidity results for generalized Bernoulli crossed products}\label{SOL}


The second ingredient needed in the proof of Theorem \ref{superr} is the following relative solidity result which generalizes \cite[Theorem 2]{CI08}.

\begin{thm}\label{solid}
Let $m\in\mathbb N$. For $1\leq j\leq m$, let $B_j\curvearrowright I_j$ be an action such that $\emph{Stab}_{B_j}(i)$ is amenable for every $i\in I_j$. Let also
 $(\A_j,\tau)$ be an abelian tracial von Neumann algebra, $B_j\curvearrowright (\A_j^{I_j},\tau)$ be a trace preserving action built over $B_j\curvearrowright I_j$ and put $\M_j=\A_j^{I_j}\rtimes B_j$.
Denote $\A=\,\overline{\otimes}\,_{j=1}^m\A_j^{I_j}$ and $\M=\,\overline{\otimes}\,_{j=1}^m\M_j$. Let $\Q\subset p\A p$ be a von Neumann subalgebra.
Assume that $\Q\nprec_{\A}\,\overline{\otimes}\,_{j\not=k}\A_j^{I_j}$, for every $1\leq k\leq m$.

Then $\Q'\cap p\M p$ is amenable.
\end{thm}

To prove Theorem \ref{solid} we rely on a corollary of \cite[Theorem 4.2 and Corollary 4.3]{IPV10}:

\begin{cor}[\cite{IPV10}]\label{rigsub2}
Let $B\curvearrowright I$ be an action such that  $\emph{Stab}_B(i)$ is amenable for every $i\in I$.  Let $(\A,\tau)$ be an abelian tracial von Neumann algebra and $B\curvearrowright^\sigma (\A^I,\tau)$ be a trace preserving action built over $B\curvearrowright I$.
Denote $\M=\A^I\rtimes B$ and let $(\Nn,\tau)$ be a tracial von Neumann algebra.
Let $\Q\subset p(\A^I\,\overline{\otimes}\,\Nn)p$ be a von Neumann subalgebra such that $\Q'\cap p(\M\,\overline{\otimes}\,\Nn)p$ is strongly nonamenable relative to $1\,\overline{\otimes}\,\Nn$ inside $\M\,\overline{\otimes}\,\Nn$.

Then $\Q\prec_{\A^I\,\overline{\otimes}\,\Nn}1\,\overline{\otimes}\,\Nn$.
\end{cor}

\begin{proof} Let $\P=\mathcal N_{p(\M\,\overline{\otimes}\,\Nn)p}(\Q)''$.
As $\P$ contains $\Q'\cap p(\M\,\overline{\otimes}\,\Nn)p$, it is strongly nonamenable relative to $1\,\overline{\otimes}\,\Nn$.  Let $i\in I$.
Since $\A$ is abelian and $\text{Stab}_B(i)$ is amenable,  $\A^I\rtimes\text{Stab}_B(i)$ is amenable and thus $(\A^I\rtimes\text{Stab}_B(i))\,\overline{\otimes}\,\Nn$ is amenable relative to $1\,\overline{\otimes}\,\Nn$.
By \cite[Proposition 2.4]{OP07} we derive that $\P$ is strongly nonamenable relative to $(\A^I\rtimes\text{Stab}_B(i))\,\overline{\otimes}\,\Nn$. In particular, using \cite[Lemmas 2.4 and 2.6]{DHI16}, we get that $\P\nprec_{\M\,\overline{\otimes}\,\Nn}(\A^I\rtimes\text{Stab}_B(i))\,\overline{\otimes}\,\Nn$.

Assume that $\sigma$ is the generalized Bernoulli action associated to $B\curvearrowright I$. Then the proof of \cite[Corollary 4.3]{IPV10} shows that $(\star)$  {$\tau(u^*(\theta_\rho\otimes\text{id})(u))\geq\delta$}, for all $u\in\sU(\Q)$, for some $\rho\in (0,1)$ and $\delta>0$, where $(\theta_\rho)_{\rho\in (0,1)}$ is the tensor length deformation of $\M$. Using $(\star)$,
the proof of \cite[Theorem 4.2]{IPV10} shows that (a) $\Q\prec_{\M\,\overline{\otimes}\,\Nn}(\A^{\mathcal F}\rtimes\text{Stab}_B(\mathcal F))\,\overline{\otimes}\,\Nn$, for a finite nonempty set $\mathcal F\subset I$, or (b) $\Q\prec_{\M\,\overline{\otimes}\,\Nn}\text{L}(B)\,\overline{\otimes}\,\Nn$. Since $\Q\subset\A^I\,\overline{\otimes}\,\Nn$, (a) implies that (c) $\Q\prec_{\A^I\,\overline{\otimes}\,\Nn}\A^{\mathcal F}\,\overline{\otimes}\,\Nn$, for a finite  nonempty set $\mathcal F\subset I$, and (b) implies that (d) $\Q\prec_{\A^I\,\overline{\otimes}\,\Nn}1\,\overline{\otimes}\,\Nn$. If (d) holds, then we have the desired conclusion. Otherwise, if (c) holds but (d) fails, then arguing as in the proof of \cite[Theorem 4.2]{IPV10} (first paragraph of page 250) shows that $\P\prec_{\M\,\overline{\otimes}\,\Nn}(\A^I\rtimes\text{Stab}_B(i))\,\overline{\otimes}\,\Nn$, for some $i\in I$, which gives a contradiction.

In general, assume that $\sigma$ is built over $B\curvearrowright I$.
Let $\widetilde\A=\A*\text{L}(\mathbb Z)$. Define  the action $B\curvearrowright^{\widetilde\sigma}\widetilde\A^I$ built over $B\curvearrowright I$ whose restriction to $\A^I$ is $\sigma$  and whose restriction to $\text{L}(\mathbb Z)^I$ is the generalized Bernoulli action associated to $B\curvearrowright I$. Let $\widetilde\M=\widetilde\A^I\rtimes B$.
We claim that the  $\M$-bimodule $\mathcal H=\text{L}^2(\widetilde\M)\ominus\text{L}^2(\M)$ is weakly contained in the coarse $\M$-bimodule, $\text{L}^2(\M)\otimes\text{L}^2(\M)$.
Assuming the claim, the proof of \cite[Corollary 4.3]{IPV10} shows that $(\star)$ holds. Then the above proof extends verbatim from generalized Bernoulli actions to actions built over $B\curvearrowright I$ to give the conclusion.

To justify the claim, let  $u$  be the canonical generating unitary of $\text{L}(\mathbb Z)$. Let $\mathcal V$ be the set of unit vectors $\xi\in \widetilde\A$ of the form $\xi=u^{n_1}a_1u^{n_2}\cdots a_{k-1}u^{n_k}$, where $n_1,\cdots, n_k\in\mathbb Z\setminus\{0\}$ and $a_1,\cdots, a_{k-1}\in\A\ominus\mathbb C1$.  Let $\mathcal U$ be the set of $\eta\in\widetilde\A^I$ of the form $\eta=\Big(\bigotimes_{i\in F}\xi_i\Big)\otimes \Big(\bigotimes_{i\in I\setminus F}1\Big)$, where $F\subset I$ is finite nonempty and $(\xi_i)_{i\in F}\subset\mathcal V$. Then  for all $a,b\in \A^I, g,h\in B$ we have $$\langle au_g\eta bu_h,\eta\rangle=\tau(\widetilde\sigma_g(\eta)\eta^*)\tau(\text{E}_{\A^{I\setminus F}\rtimes\text{Stab}_B(F)}(au_g)bu_h).$$ This implies that the $\M$-bimodule $\overline{\M\eta\M}$ is a subbimodule of $\langle\M,e_{\A^{I\setminus F}\rtimes\text{Stab}_B(F)}\rangle\otimes_\M \mathcal K$, where $\mathcal K$ is the $\M$-bimodule associated to the unital completely positive map on $\M$ given by {$au_g\mapsto \tau(\widetilde\sigma_g(\eta)\eta^*)au_g$}. Since $\A$ is abelian and $\text{Stab}_B(F)$ is amenable,  $\A^{I\setminus F}\rtimes\text{Stab}_B(F)$  is amenable and thus $\overline{\M\eta\M}$  is weakly contained in the coarse $\M$-bimodule.  Since $\mathcal H$ is isomorphic to an $\M$-subbimodule of $\bigoplus_{\eta\in\mathcal U}\overline{\M\eta\M}$, the claim follows.
\end{proof}

\begin{proof}[Proof of Theorem \ref{solid}]
Assume  that $\R=\Q'\cap p\M p$ is not amenable. For $1\leq k\leq m$, let $\widehat{\A}_k=\,\overline{\otimes}\,_{j\not=k}\A_j^{I_j}$ and $\widehat{\M}_k=\,\overline{\otimes}\,_{j\not=k}\M_j$. Then the algebras $(\widehat{\M}_k)_{1\leq k\leq m}$ are in a commuting square position and we have $\cap_{k=1}^m\widehat{\M}_k=\mathbb C1$.
By \cite[Proposition 2.7]{PV11}, there is $1\leq k\leq m$ such that $\R$ is not amenable relative to $\widehat{\M}_k$.  Using \cite[Lemma 2.6(2)]{DHI16} we find a nonzero projection $p_0\in\sZ(\R'\cap p\M p)\subset\sZ(\R)$ such that $\R p_0=(\Q p_0)'\cap p_0\M p_0$ is strongly nonamenable relative to $\widehat{\M}_k$.  By applying Corollary \ref{rigsub2} we deduce that
$\Q \prec_{\M} \widehat{\M}_k$.

On the other hand, since $\Q \nprec_{\A} \widehat{\A}_k$, we can find a sequence of unitaries $u_n\in\sU(\Q)$ such that $\|\text{E}_{\widehat{\A}_k}(a^*u_nb)\|_2\rightarrow 0$, for all $a,b\in\A$.
We claim that \begin{equation}\label{noemb}\text{$\|\text{E}_{\widehat{\M}_k}(x^*u_ny)\|_2\rightarrow 0$, for all $x,y\in\M.$}\end{equation}
To prove \eqref{noemb}, we may assume that $x,y\in \M_k$ and moreover that $x=au_g,y=bu_h$, for some $a,b\in \A_k^{I_k}$ and $g,h\in B_k$. Then since $u_n\in\Q\subset\A$,  for all $n\in\mathbb N$, we have $$\|\text{E}_{\widehat{\M}_k}(x^*u_ny)\|_2=\delta_{g,h}\|\text{E}_{\widehat{\M}_k}(a^*u_nb)\|_2=\delta_{g,h}\|\text{E}_{\widehat{\A}_k}(a^*u_nb)\|_2\rightarrow 0.$$
This proves \eqref{noemb}, which contradicts that $\Q \prec_{\M} \widehat{\M}_k$ and finishes the proof.
 \end{proof}

In the proof of Theorem \ref{superr} we will in fact need the following corollary of Theorem \ref{solid}.

\begin{cor}\label{solidity}
Let $m\in\mathbb N$. For $1\leq j\leq m$, let $G_j\in\mathcal W\mathcal R(A_j,B_j\curvearrowright I_j)$, where $A_j$ is an abelian group and
$B_j\curvearrowright I_j$ an action such that $\emph{Stab}_{B_j}(i)$ is amenable for every $i\in I_j$ and $\{i\in I_j\mid g\cdot i\not=i\}$ is infinite for every $g\in B_j\setminus\{1\}$.
Define $G=\oplus_{j=1}^mG_j$ and $A=\oplus_{j=1}^mA_j^{(I_j)}$.
Let $\Q\subset  p(\emph{L}(G)\,\overline{\otimes}\,\mathbb M_n(\mathbb C)) p$ be a von Neumann subalgebra such that $\Q\prec^{\emph{s}}_{\emph{L}(G)\,\overline{\otimes}\,\mathbb M_n(\mathbb C)}\emph{L}(A)\,\overline{\otimes}\,\mathbb M_n(\mathbb C)$ and $ \Q\nprec_{\emph{L}(G)\,\overline{\otimes}\,\mathbb M_n(\mathbb C)} \emph{L}(\oplus_{j\not=k}A_j^{{(I_j)}})\,\overline{\otimes}\,\mathbb M_n(\mathbb C)$, for all $1\leq k\leq m$.

Then $\Q'\cap p(\emph{L}(G)\,\overline{\otimes}\,\mathbb M_n(\mathbb C))p$ is amenable.
\end{cor}

Corollary \ref{solidity} is obtained by combining Theorem \ref{solid} with the following ``transfer" lemma.

\begin{lem}\label{transfer}
Let $A$ be a normal abelian subgroup of a countable group $G$. Assume that $\{aga^{-1}\mid a\in A\}$ is infinite, for every $g\in G\setminus A$.  Consider the action of $G/A$ on $A$ by conjugation: $g\cdot a=\widehat{g}a\widehat{g}^{-1}$, for every $g\in G/A$ and $a\in A$, where $\varepsilon\colon G\rightarrow G/A$ denotes the quotient homomorphism and $\widehat{g}\in G$ is any element such that $\varepsilon(\widehat{g})=g$. Define the semidirect product group $H=A\rtimes G/A$.
Let $\Q\subset p\emph{L}(A)p$ be a von Neumann subalgebra.

If $\Q'\cap p\emph{L}(H)p$ is amenable, then $\Q'\cap p\emph{L}(G)p$ is amenable.
\end{lem}

\begin{proof} Denote $\M=\text{L}(G)$, $\Nn=\text{L}(H)$ and $\P=\text{L}(A)$.
Identify  $\P=\text{L}^{\infty}(X,\mu)$ and consider the associated measure preserving action $G/A\curvearrowright (X,\mu)$, where $X$ denotes the dual of $A$ endowed with its Haar measure $\mu$.
Since $\{aga^{-1}\mid a\in A\}$ is infinite, for every $g\in G\setminus A$, we get that $\P$ is a Cartan subalgebra of $\M$.
Moreover, $\sR(\P\subset \M)$ can be identified with
$\sR =\sR(G/A\curvearrowright X)$. Thus, we get that $\M=\text{L}_{c}(\sR)$, for a $2$-cocycle $c\in\text{H}^2(\sR,\mathbb T)$. {We endow $\sR$ with the usual Borel measure $\widetilde\mu$ given by $\widetilde\mu(\mathcal T)=\int_X|\{y\in X\mid (x,y)\in \mathcal T\}|\;\text{d}\mu(x)$, for every Borel subset $\mathcal T\subset\sR$.}

We continue  by repeating part of the proof of \cite[Proposition 6]{CI08}.
Let $X_0\subset X$ be a measurable set such that $p={\bf 1}_{X_0}$. Endow $X_0$ with the probability measure $\mu(X_0)^{-1}\mu_{|X_0}$. Since $\Q\subset p\P p$ is a von Neumann subalgebra, there are a standard probability space $(Z,\rho)$ and a measurable, measure preserving onto map $\pi\colon X_0\rightarrow Z$ such that  we have $\Q=\{f\circ\pi\mid f\in \text{L}^{\infty}(Z,\rho)\}$. Since $p\P p\subset \Q'\cap p\M p$, \cite[Proposition 6.1]{Dy63} implies that $p\P p$ is a Cartan subalgebra of $\Q'\cap p\M p$. By \cite[Theorem 1]{FM77a}, we can find a subequivalence relation $\sS$ of $\sR|X_0:=\sR\cap (X_0\times X_0)$ such that $\Q'\cap p\M p=\text{L}_{d}(\sS)$, where $d\in\text{H}^2(\sS,\mathbb T)$ denotes the restriction of $c$ to $\sS$.

We claim that $\sS=\{(x_1,x_2)\in \sR|X_0\mid \pi(x_1)=\pi(x_2)\}$, {$\widetilde\mu$-almost everywhere}. To see this, let $\varphi\in [\sR|X_0]$. Then $\varphi\in [\sS]$ if and only if $u_{\varphi}\in \text{L}_d(\sS)$, that is, if and only if $u_{\varphi}$ commutes with $\Q$. The latter is equivalent to having for every $f\in \text{L}^{\infty}(Z,\rho)$ that $f(\pi(\varphi(x)))=f(\pi(x)))$, for almost every $x\in X_0$. Thus, $\varphi\in [\sS]$ if and only if $\pi(\varphi(x))=\pi(x)$, for almost every $x\in X_0$, which proves our claim.

Finally, note that  $\Nn=\text{L}(\sR)$ and arguing as in the previous paragraph shows that $\Q'\cap p\Nn p=\text{L}(\sS)$. If $\text{L}(\sS)$ is amenable, by Connes-Feldman-Weiss' theorem \cite{CFW81} we get that $\sS$ is an amenable and thus hyperfinite equivalence relation. This gives that $\Q'\cap p\M p=\text{L}_{d}(\sS)$ is a hyperfinite and thus amenable von Neumann algebra.
\end{proof}

\begin{proof}[Proof of Corollary \ref{solidity}]
Using a standard argument, whose proof we leave to the reader, the conclusion reduces to the following claim:  if $\Q\subset p\text{L}(A) p$ is a von Neumann subalgebra such that $\Q\nprec_{\text{L}(A)}\text{L}(\oplus_{j\not=k}A_j^{(I_j)})$, for every $1\leq k\leq m$, then $\Q'\cap p\text{L}(G) p$ is amenable.

Let $\Q\subset p\text{L}(A) p$ be a von Neumann subalgebra such that $\Q\nprec_{\text{L}(A)}\text{L}(\oplus_{j\not=k}A_j^{(I_j)})$, for every $1\leq k\leq m$.
 Consider the conjugation action of $B=G/A$ on $A$ and define $H=A\rtimes B$.
For $1\leq j\leq m$, consider the conjugation action of $B_j=G_j/A_j^{(I_j)}$ on $A_j^{(I_j)}$. By Remark \ref{bernoulli} the associated trace preserving action $B_j\curvearrowright \text{L}(A_j^{(I_j)})$ is built over $B_j\curvearrowright I_j$. Since $H=\oplus_{j=1}^m(A_j^{(I_j)}\rtimes B_j)$ we have $\text{L}(H)=\,\overline{\otimes}\,_{j=1}^m\big(\text{L}(A_j^{(I_j)})\rtimes B_j\big)$.
Since $\text{Stab}_{B_j}(i)$ is amenable, for all $i\in I_j$ and $1\leq j\leq k$, by applying Theorem \ref{solid} we get that $\Q'\cap p\text{L}(H)p$ is amenable.

Next, let $g=(g_1,\cdots,g_m)\in G\setminus A$. Then  $g_j\in G_j\setminus A_j^{(I_j)}$, for some $1\leq j\leq m$.
If $\varepsilon_j\colon G_j\rightarrow B_j$ is the quotient homomorphism, then $\varepsilon_j(g_j)\not=1$, hence $\{i\in I_j\mid \varepsilon_j(g_j)\cdot i\not=i\}$ is infinite. This implies that  $\{bg_jb^{-1}\mid b\in A_j^{(I_j)}\}$ is infinite. Thus, $\{aga^{-1}\mid a\in A\}$ is infinite. Since this holds for every $g\in G\setminus A$, we can apply Lemma \ref{transfer} to deduce that $\Q'\cap p\text{L}(G)p$ is amenable, as claimed.
\end{proof}


\subsection{Proof of Theorem \ref{superr}}\label{SUPER}


In preparation for the proof of Theorem \ref{superr}, we introduce some notation and record three useful facts.
Let $A$ be a nontrivial abelian group, $B$ a nontrivial ICC subgroup of a hyperbolic group and $B\curvearrowright I$ an action with $\text{Stab}_B(i)$ is amenable, for every  $i\in I$.
Let $G\in\mathcal W\mathcal R(A,B\curvearrowright I)$ be a  property (T) group.
Denote $\M=\text{L}(G)$ and assume that
$\M^t=\text{L}(H)$, for a countable group $H$ and $t>0$.

Let $\Delta_0\colon \text{L}(H)\rightarrow \text{L}(H)\,\overline{\otimes}\,\text{L}(H)$ be the comultiplication given by $\Delta_0(v_h)=v_h\otimes v_h$, for every $h\in H$. Let $n$ be the smallest integer such that $n\geq t$.
Denote $\sM=\M\,\overline{\otimes}\,\M\,\overline{\otimes}\,\mathbb M_n(\mathbb C)$.
Then $\Delta_0$ can be amplified to a unital $*$-homomorphism $\Delta\colon \M\rightarrow p\sM p$, where $p\in\sM$ is a projection with $(\tau\otimes\tau\otimes\text{Tr})(p)=t$.

\begin{rem}
Assume that $t\in\mathbb N$, so that $n=t$, $p=1$ and $\text{L}( H)=\M^n=\M\,\overline{\otimes}\,\mathbb M_n(\mathbb C)$.
For further reference, we make explicit the construction of $\Delta$ in terms of $\Delta_0$. To this end, let $\psi\colon \sM\,\overline{\otimes}\,\mathbb M_n(\mathbb C)\rightarrow \M\,\overline{\otimes}\,\mathbb M_n(\mathbb C)\,\overline{\otimes}\,\M\,\overline{\otimes}\,\mathbb M_n(\mathbb C)=\text{L}( H)\,\overline{\otimes}\,\text{L}( H)$ be the $*$-isomorphism given by $\psi(a\otimes b\otimes c\otimes d)=a\otimes c\otimes b\otimes d$, for every $a,b\in \M$ and $c,d\in\mathbb M_n(\mathbb C)$. Let $U\in \M\,\overline{\otimes}\,\mathbb M_n(\mathbb C)\,\overline{\otimes}\,\M\,\overline{\otimes}\,\mathbb M_n(\mathbb C)$ be a unitary such that $\Delta_0(1_\M\otimes x)=U\psi(1_{\sM}\otimes x)U^*$, for every $x\in\mathbb M_n(\mathbb C)$. Then $\psi^{-1}\circ\text{Ad}(U^*)\circ\Delta_0\colon \M\,\overline{\otimes}\,\mathbb M_n(\mathbb C)\rightarrow\sM\,\overline{\otimes}\,\mathbb M_n(\mathbb C)$ is a unital $*$-homomorphism which leaves $1\,\overline{\otimes}\,\mathbb M_n(\mathbb C)$ fixed, so it can be written as $\Delta\otimes\text{Id}$, where $\Delta\colon \M\rightarrow\sM$ is the desired unital $*$-homomorphism that amplifies $\Delta_0$. Thus, we conclude
 \begin{equation}\label{Delta}
\Delta_0=\text{Ad}(U)\circ\psi\circ(\Delta\otimes\text{Id})
 \end{equation}
In the proof of Theorem \ref{superr}, we will combine \eqref{Delta} with the symmetry and associativity properties of $\Delta_0$:  ${\nabla}\circ\Delta_0=\Delta_0$ and $(\Delta_0\otimes\text{Id})\circ\Delta_0=(\text{Id}\otimes\Delta_0)\circ\Delta_0$, where ${\nabla}$ is the flip automorphism of $\text{L}( H)\,\overline{\otimes}\,\text{L}( H)$ given by ${\nabla}(x\otimes y)=y\otimes x$, for every $x,y\in \text{L}( H)$.
\end{rem}

\begin{lem}[\cite{IPV10}]\label{comultiply} Let $\Delta\colon \M\rightarrow p\sM p$ be as defined above. Then the following hold:
\begin{enumerate}
\item[(a)] $\Delta(\Q)\nprec_{\sM}\M\,\overline{\otimes}\, 1\,\overline{\otimes}\,\mathbb M_n(\mathbb C)$ and $\Delta(\Q)\nprec_{\sM}1\,\overline{\otimes}\, \M\,\overline{\otimes}\,\mathbb M_n(\mathbb C)$, for any diffuse von Neumann subalgebra $\Q\subset \M$.
\item[(b)] $\Delta(\M)\nprec_{\sM}\M\,\overline{\otimes}\,\emph{L}(G_0)\,\overline{\otimes}\,\mathbb M_n(\mathbb C)$ and $\Delta(\M)\nprec_{\sM}\emph{L}(G_0)\,\overline{\otimes}\,\M\,\overline{\otimes}\,\mathbb M_n(\mathbb C)$, for any infinite index subgroup $G_0<G$.
\item[(c)] If $\mathcal H\subset \emph{L}^2(p\sM p)$ is a $\Delta(\M)$-sub-bimodule which is right finitely generated, then we have  $\mathcal H\subset \emph{L}^2(\Delta(\M))$.

\end{enumerate}
\end{lem}

\begin{proof} (a)\; This part is \cite[Proposition 7.2(1)]{IPV10}.

(b)\; If  $\Delta(\M)\prec_{\sM}\M\,\overline{\otimes}\,\text{L}(G_0)\,\overline{\otimes}\,\mathbb M_n(\mathbb C)$, for a subgroup $G_0<G$, then the proof of \cite[Proposition 7.2(2)]{IPV10} shows that $\M\prec_{\M}\text{L}(G_0)$ and so $G_0<G$ has finite index. Similarly,  $\Delta(\M)\prec_{\sM}\text{L}(G_0)\,\overline{\otimes}\,\M\,\overline{\otimes}\,\mathbb M_n(\mathbb C)$ also implies that $G_0<G$ has finite index.

(c)\; Since $G$ is  ICC, $\text{C}_G(g)<G$ has infinite index and thus $\M\nprec_{\M}\text{L}(\text{C}_G(g))$, for every $g\in G\setminus\{e\}$. The conclusion then follows from \cite[Proposition 7.2(3)]{IPV10}.\end{proof}

In the proof of Theorem \ref{superr} we will also need the fact that the set $\{i\in I\mid b\cdot i\not=i\}$ is infinite, for every $b\in B\setminus \{1\}$. This more generally holds if $B$ is acylindrically hyperbolic:

\begin{lem}\label{infinite}
Let $B$ be an \emph{ICC} group acting on a set $I$. Then the following hold:
\begin{enumerate}
\item[(a)]
Assume that $B$ is acylindrically hyperbolic and $\emph{Stab}_B(i)$ is amenable, for every $i\in I$. Then,  for every non-trivial $b\in B$, the set $\{ i\in I\mid b \cdot i\ne i\}$ is infinite.
\item[(b)]
Assume that $B\cdot i$ is infinite, for every $\in I$. Let $A$ be a group. Then every $G\in\WR(A,B\curvearrowright I)$ is \emph{ICC}.
\end{enumerate}
\end{lem}

\begin{proof} (a)\;
Suppose that the set $\{ i\in I\mid b \cdot i\ne i\}$ is finite for some  $b\in B$. Let $N$ denote the minimal normal subgroup of $B$ containing $b$. The subgroup $N$ is generated by the set $X=\{ t^{-1}bt\mid t\in B\}$. For every finite subset $F\subseteq X$, the subgroup $\langle F\rangle $ stabilizes all but finitely many elements of $I$. Since acylindrically hyperbolic groups are non-amenable, $I$ must be infinite. In particular, the subgroup $\langle F\rangle $ stabilizes at least one element of $I$ and hence it is amenable. This implies that $N$ is amenable being the union of amenable groups. By \cite[Corollary 8.1 (a)]{Osi16}, the amenable radical of every acylindrically hyperbolic group is finite. Since $B$ is ICC, $N$ must be trivial. Thus, $b=1$.

(b)\; Let $G\in\WR(A,B\curvearrowright I)$, for some group $A$, and denote by $\varepsilon:G\rightarrow B$ the quotient homomorphism. Let $g\in G\setminus \{1\}$. We treat two cases. First, assume that $\varepsilon(g)\not=1$. Since $B$ is ICC, we get that $\{\varepsilon(hgh^{-1})\mid h\in G\}=\{b\varepsilon(g)b^{-1}\mid b\in B\}$ is infinite, and thus the conjugacy class $\{hgh^{-1}\mid h\in G\}$ is infinite. Second, assume that $\varepsilon(g)=1$, i.e., $g\in A^{(I)}\setminus\{1\}$.
If $a=(a_i)_{i\in I}\in A^{(I)}\setminus\{1\}$ we denote by $\text{supp}(a)=\{i\in I\mid a_i\not=1\}$ the support of $a$. Let $i\in\text{supp}(g)$. If $j\in B\cdot i$, then we can write $j=\varepsilon(h)\cdot i$, for some $h\in G$. Since $\text{supp}(hgh^{-1})=\varepsilon(h)\cdot\text{supp}(g)$, we get that $j\in\text{supp}(hgh^{-1})$.
Thus, $B\cdot i\subset\cup_{h\in G}\text{supp}(hgh^{-1})$. Since $B\cdot i$ is infinite, we conclude that the conjugacy class $\{hgh^{-1}\mid h\in G\}$ is infinite as well. This finishes the proof.
\end{proof}

\begin{proof}[Proof of Theorem \ref{superr}]  Let $\pi\colon G\rightarrow  B$ be the quotient homomorphism.
For $g\in B$, fix $\widehat{g}\in G$ with $\pi(\widehat{g})=g$.
Let $K$ be a hyperbolic group containing $B$ and denote still by $\pi$ the homomorphism $\pi\colon G\rightarrow K$.
 Let $\mathbb D_n(\mathbb C)\subset\mathbb M_n(\mathbb C)$ be the subalgebra of diagonal matrices. For $1\leq i\leq n$, let $e_i={\bf 1}_{\{i\}}\in\mathbb D_n(\mathbb C)$.
Denote $$\text{$\P=\text{L}(A^{(I)})$,\;\;  $\sP=\P\,\overline{\otimes}\,\P\,\overline{\otimes}\,\mathbb D_n(\mathbb C)$\;\; and \;\;
$\Q=\Delta(\P)'\cap p\sM p$.}$$

The proof is divided into six steps.

{\it Step 1.} $\Q\prec_{\sM}^{\text s}\sP$.

\begin{proof}   We first prove that $\Delta(\P)\prec_{\sM}^{\text s}\sP$.  Write $\sM=(\M\,\overline{\otimes}\, 1\,\overline{\otimes}\,\mathbb M_n(\mathbb C))\rtimes G$, using the trivial action of $G$.  As $\ker(\pi)=A^{(I)}$, we have $(\M\,\overline{\otimes}\, 1\,\overline{\otimes}\,\mathbb M_n(\mathbb C))\rtimes \ker(\pi)=\M\,\overline{\otimes}\,\P\,\overline{\otimes}\,\mathbb M_n(\mathbb C)$.
Since  $\Delta(\P)$ is amenable,    $\Delta(\M)\subset\sN_{p\sM p}(\Delta(\P))''$ has property (T) and $\Delta(\M)\nprec_{\sM} \M\,\overline{\otimes}\,\P\,\overline{\otimes}\,\mathbb M_n(\mathbb C)$ by Lemma \ref{comultiply} (b),  from Theorem \ref{relativeT} we derive that $\Delta(\P)\prec_{\sM}^{\text s}\M\,\overline{\otimes}\,\P\,\overline{\otimes}\,\mathbb M_n(\mathbb C)$.
Similarly,  $\Delta(\P)\prec_{\sM}^{\text s}\P\,\overline{\otimes}\,\M\,\overline{\otimes}\,\mathbb M_n(\mathbb C)$. Combining these facts with \cite[Lemma 2.8(2)]{DHI16} gives that $\Delta(\P)\prec_{\sM}^{\text s}\P\,\overline{\otimes}\,\P\,\overline{\otimes}\,\mathbb M_n(\mathbb C)$, which proves our claim.

Next, we have that $\Delta(\P)\nprec_{\sM} \P\,\overline{\otimes}\, 1\,\overline{\otimes}\, \mathbb M_n(\mathbb C)$ and $\Delta(\P)\nprec_{\sM} 1\,\overline{\otimes}\, \P\,\overline{\otimes}\, \mathbb M_n(\mathbb C)$ by Lemma \ref{comultiply} (a). Since the action $B\curvearrowright I$ has amenable stabilizers, $\{i\in I\mid b\cdot i\not=i\}$ is infinite, for every $b\in B\setminus\{1\}$, by Lemma \ref{infinite} (a). Thus, using that $\Delta(\P)\prec_{\sM}^{\text s}\sP$, Corollary \ref{solidity} implies that  $\Q$ is amenable.

Finally, since  $\Q$ is amenable and $\Delta(\M)\subset\sN_{p\sM p}(\Q)''$, repeating the first paragraph of the proof with $\Q$ instead of $\Delta(\P)$ completes the proof of this step.
\end{proof}

As $\sP\subset\sM$ is a Cartan subalgebra, by combining Step 1 with Lemma \ref{conj1}, after replacing $\Delta$ with $\text{Ad}(u)\circ\Delta$, for some $u\in\sU(\sM)$, we may assume that $p\in \sP$ and \begin{equation}\label{deltaP}\Delta(\P)\subset \sP p\subset\Q.\end{equation}

If $g\in B$, then $\Delta(u_{\widehat{g}})$ normalizes $\Delta(\P)$ and thus $\Q$. Denote $\sigma_g=\text{Ad}(\Delta(u_{\widehat{g}}))\in\text{Aut}(\Q)$.
Since $\widehat{g}\widehat{h}{\widehat{(gh)}}^{-1}\in\ker(\pi)=A^{(I)}$, we have $\Delta(u_{\widehat{g}})\Delta(u_{\widehat{h}})\Delta(u_{\widehat{gh}})^*\in\Delta(\P)$, for every $g,h\in B$.  Since $\Delta(\P)\subset\sZ(\Q)$, $\sigma =(\sigma_g)_{g\in B}$ defines an action of $B$ on $\Q$ which leaves $\Delta(\P)$ invariant. Since the restriction of $\sigma$ to $\Delta(\P)$ is conjugate to an action $B\curvearrowright \text{L}(A)^I$ built over $B\curvearrowright I$ (see Remark \ref{bernoulli}), it is free and weakly mixing. This fact can be strengthened as follows.

{\it Step 2.} The action $B\curvearrowright^{\sigma}\Q$ is weakly mixing.

\begin{proof}  This claim is a consequence of Step 3 in the proof of \cite[Theorem 8.2]{IPV10}, which we recall for completeness.
Let $\mathcal H\subset \text{L}^2(\Q)$ be a finite dimensional $\sigma(B)$-invariant subspace.
Let $\mathcal K\subset \text{L}^2(p\sM p)$ be the $\|\cdot\|_2$-closure of the linear span of $\mathcal H\Delta(\M)$. Since $\mathcal H$ and $\Delta(\P)$ commute, we get that $\Delta(\P)\mathcal K=\mathcal K$. If $g\in B$, then $\Delta(u_{\widehat{g}})\mathcal H=\mathcal H\Delta(u_{\widehat{g}})$ and so $\Delta(u_{\widehat{g}})\mathcal K=\mathcal K$. Since $G=\{a\widehat{g}\mid a\in A^{(I)},g\in B\}$, $\mathcal K$ is a left $\Delta(\M)$-module. Thus, $\mathcal K$ is a $\Delta(\M)$-bimodule which is right finitely generated as $\mathcal H$ is finite dimensional. Lemma \ref{comultiply} (c) gives that $\mathcal K\subset \text{L}^2(\Delta(\M))$, hence $\mathcal H\subset \text{L}^2(\Delta(\M))$. Since $\mathcal H$ commutes with $\Delta(\P)$, we have $\mathcal H\subset \text{L}^2(\Delta(\P))$. As the restriction of $\sigma$ to $\Delta(\P)$ is weakly mixing, we conclude that $\mathcal H\subset\mathbb Cp$, as claimed.
\end{proof}

Steps 1 and 2 imply that $\Q$ is a type I$_k$ algebra, for some $k\in\mathbb N$.
Using \eqref{deltaP}, the beginning of the proof of Lemma \ref{conj2} shows that there is a decomposition $\Q=\sZ(\Q)\,\overline{\otimes}\,\mathbb M_k(\mathbb C)$ such that  $\sP p=\sZ(\Q)\,\overline{\otimes}\,\mathbb D_k(\mathbb C)$. Therefore,  $(\Q)_1\subset\sum_{i=1}^k(\sP p)_1x_i$, for some $x_1,\cdots,x_k\in \Q$.
Moreover, by Lemma \ref{conj2} there is an action $\beta=(\beta_g)_{g\in B}$ of $B$ on $\Q$ such that

\begin{itemize}
\item  for every $g\in B$ we have that $\beta_g=\sigma_g\circ\text{Ad}(\omega_g)=\text{Ad}(\Delta(u_{\widehat{g}})\omega_g)$, for some $\omega_g\in\sU(\Q)$,
\item $\sP p$ is $\beta(B)$-invariant and the restriction of $\beta$ to $\sP p$ is free, and
\item the minimal projections $p_1,\cdots,p_k$ of $1\,\overline{\otimes}\,\mathbb D_k(\mathbb C)\subset\mathcal Pp$ are $\beta(B)$-invariant and the restriction of $\beta$ to $\sP p_i$ is weakly mixing, for every $1\leq i\leq k$.
\end{itemize}

Our next goal is to apply Theorem \ref{SOE}.
Let $(Y,\nu)$ be the dual of $A$ with its Haar measure.
Let $(X,\mu)=(Y^{I}\times Y^{I},\nu^{ I}\times\nu^{I})$ and $(\widetilde X,\widetilde\mu)=(X\times\mathbb Z/n\mathbb Z, \mu\times c)$, where $c$ is the counting measure of $\mathbb Z/n\mathbb Z$. Identify  $\P=\text{L}^{\infty}(Y^I)$ and $\sP=\text{L}^{\infty}(\widetilde X)$.
Consider the action $B\curvearrowright^{\alpha_0} (Y^I,\nu^I)$ given by $\alpha_0(g)=\text{Ad}(u_{\widehat{g}})$, for all $g\in B$. By Remark \ref{bernoulli}, $\alpha_0$ is conjugate to an action built over $B\curvearrowright I$.
Let $ B\times B\curvearrowright^{\alpha}(X,\mu)$ be  given by $(g_1,g_2)\cdot (x_1,x_2)=(g_1\cdot x_1,g_2\cdot x_2)$, for all $g_1,g_2\in B$ and $x_1,x_2\in Y^{I}$. Let $ B \times B\times\mathbb Z/n\mathbb Z\curvearrowright^{\widetilde\alpha} (\widetilde X,\widetilde\mu)$ be the action given by $(g,a)\cdot (x,b)=(g\cdot x,a+b)$.
Let $X_0\subset \widetilde X$ be a measurable set such that $p={\bf 1}_{X_0}$.
Since $\sP p=\text{L}^{\infty}(X_0)$ is $\beta(B)$-invariant, we get a measure preserving action $B\curvearrowright^\beta (X_0,\widetilde\mu_{|X_0})$.

Since the restriction of $\beta$ to $\sP p$ is implemented by unitaries in $p\sM p$ and we have that $\sR(\sP\subset\sM)=\sR( B\times  B\times\mathbb Z/n\mathbb Z\curvearrowright^{\widetilde\alpha}\widetilde X)$,
we deduce that
\begin{equation}\label{include}
\text{$\beta( B)\cdot x\subset\widetilde\alpha( B\times B\times\mathbb Z/n\mathbb Z)\cdot x$, for almost every $x\in X_0$.}
\end{equation}

In order to apply Theorem \ref{SOE} to \eqref{include}, we first establish the following claim:

{\it Step 3}. Let $ B_0<B$ be an infinite index subgroup. Then there is a sequence $(h_m)\subset B$ such that  for every $s,t\in B$ we have $\widetilde\mu(\{x\in X_0\mid \beta_{h_m}(x)\in\widetilde\alpha(sB_0t\times B\times\mathbb Z/n\mathbb Z)(x)\})\rightarrow 0$ and $\widetilde\mu(\{x\in X_0\mid \beta_{h_m}(x)\in\widetilde\alpha(B\times sB_0t\times\mathbb Z/n\mathbb Z)(x)\})\rightarrow 0$.

\begin{proof} Let $G_0=\pi^{-1}( B_0$). Then $G_0<G$ is an infinite index subgroup.
{Since $\Delta(u_g)_{g\in G}$ is a group of unitaries generating $\Delta (\mathcal M)$, using Lemma \ref{comultiply}(b) and Theorem \ref{corner}}
we can find a sequence $(k_m)\subset G$ such that  for every $x,y\in\sM$ we have
\begin{equation}\label{weak}\text{$\|E_{\M\,\overline{\otimes}\,\text{L}(G_0)\,\overline{\otimes}\,\mathbb M_n(\mathbb C)}(x\Delta(u_{k_m})y)\|_2\rightarrow 0$ and $\|E_{\text{L}(G_0)\,\overline{\otimes}\,\M\,\overline{\otimes}\,\mathbb M_n(\mathbb C)}(x\Delta(u_{k_m})y)\|_2\rightarrow 0$.}\end{equation}

We will show that $h_m=\pi(k_m)\in B$ satisfy the assertion of the claim.
Since $k_m^{-1}\widehat{h_m}\in A^{(I)}$ and $\omega_{h_m}\in\sU(\Q)$, we get that $\Delta(u_{\widehat{h_m}})\omega_{h_m}\in\Delta(u_{k_m}) \sU(\Q)$. Thus, $\Delta(u_{\widehat{h_m}})\omega_{h_m}\in\sum_{i=1}^k\Delta(u_{k_m})(\sP p)_1x_i$, for every $m\in\mathbb N$. As $\sP$ is regular in $\sM$ and contained in $\M\,\overline{\otimes}\,\text{L}(G_0)\,\overline{\otimes}\,\mathbb M_n(\mathbb C)$ and $\text{L}(G_0)\,\overline{\otimes}\,\M\,\overline{\otimes}\,\mathbb M_n(\mathbb C)$,  \eqref{weak} implies that for every $x,y\in\sM$
\begin{equation}\label{v_h_m}\text{$\|E_{\M\,\overline{\otimes}\,\text{L}(G_0)\,\overline{\otimes}\,\mathbb M_n(\mathbb C)}(x\Delta(u_{\widehat{h_m}})\omega_{h_m}y)\|_2\rightarrow 0$, $\|E_{\text{L}(G_0)\,\overline{\otimes}\,\M\,\overline{\otimes}\,\mathbb M_n(\mathbb C)}(x\Delta(u_{\widehat{h_m}})\omega_{h_m}y)\|_2\rightarrow 0$.}\end{equation}

On the other hand, we have that $\beta_{h_m}=\text{Ad}(\Delta(u_{\widehat{h_m}})\omega_{h_m})$ and $\alpha(g_1,g_2)=\text{Ad}(u_{(\widehat{g_1},\widehat{g_2})})$, for every $(g_1,g_2)\in B\times B$. These facts imply that
$\widetilde\mu(\{x\in X_0\mid  \beta_{h_m}(x)\in\widetilde{\alpha}(sB_0t\times B\times\mathbb Z/n\mathbb Z)(x)\})$ is equal to $\|E_{\text{L}(G_0)\,\overline{\otimes}\,\M\,\overline{\otimes}\,\mathbb M_n(\mathbb C)}((u_{\widehat{s}}^*\otimes 1\otimes 1)\Delta(u_{\widehat{h_m}})\omega_{h_m}(u_{\widehat{t}}^*\otimes 1\otimes 1))\|_2^2.$
Since by \eqref{v_h_m} the last term converges to $0$, as $m\rightarrow \infty$, this proves the first assertion of Step 3. The second assertion follows similarly.
\end{proof}

Next, $\alpha$ is {conjugate} to an action $ B\times B\curvearrowright (Y^J,\nu^J)$ built over $ B\times B\curvearrowright J= I\times\{1,2\}$ given by $(g_1,g_2)\cdot (i,j)=(g_j\cdot i,j)$.  Thus, $\text{Stab}_{ B\times B}(j)$ is equal to either $B_0\times B$ or $B\times B_0$, where $B_0=\text{Stab}_B(i)$. Note that since $B_0$ is amenable and $B$ is nonamenable, the inclusion $B_0<B$ has infinite index.
Since $ B$ is ICC,  for every $g\in ( B\times B)\setminus\{(1,1)\}$, there is an infinite index subgroup $ B_1< B$  such that $\text{C}_{ B\times B}(g)\subset B\times B_1$ or $\text{C}_{ B\times B}(g)\subset  B_1\times B$.

Fix $1\leq i\leq k$. Let $X_i\subset X_0$ be a $\beta( B)$-invariant measurable set such that  $p_i={\bf 1}_{X_i}$. Since $\beta_{|X_i}$ is free, weakly mixing and $ B$ has property (T), equation \eqref{include},  the previous paragraph and Step 3 show that the conditions of the moreover assertion of Theorem \ref{SOE} are satisfied by $\beta_{|X_i}$ and $\alpha$. Thus, Theorem \ref{SOE} implies that $\widetilde\mu(X_i)=1$ and there are $\theta_i\in [\sR( B\times B\times\mathbb Z/n\mathbb Z\curvearrowright^{\widetilde\alpha}\widetilde X)]$ and an injective homomorphism $\varepsilon_i=(\varepsilon_{i,1},\varepsilon_{i,2})\colon  B\rightarrow B\times B$
such that $\theta_i(X_i)=X\times\{i\}\equiv X$ and $\theta_i\circ\beta(h)_{|X_i}=\alpha(\varepsilon_i(h))\circ{\theta_i}_{|X_i}$, for every $h\in B$.

Let $u_i\in\sN_{\sM}(\sP)$ such that $u_iau_i^*=a\circ\theta_i^{-1}$, for every $a\in\sP$. Then $u_ip_iu_i^*=1\otimes 1\otimes e_i$ and the last relation implies that we can find $(\zeta_{i,h})_{h\in B}\subset\sU(\P\,\overline{\otimes}\,\P)$ such that
\begin{equation}\label{conjug}\text{$u_i\Delta(u_{\widehat{h}})\omega_hp_iu_i^*=\zeta_{i,h}u_{(\widehat{\varepsilon_{i,1}(h)},\widehat{\varepsilon_{i,2}(h)})}\otimes e_i$, for every $h\in B$.}\end{equation}

{\it Step 4.} $\varepsilon_i$ is conjugate to $\varepsilon_j$, for every $1\leq i,j\leq k$.

\begin{proof}
We claim that $ B_0=\varepsilon_{i,1}( B)$ has finite index in $ B$. If this is false, then $G_0=\pi^{-1}( B_0)$ has infinite index in $G$. On the other hand,  \eqref{conjug} gives that $\Delta(\M)\prec_{\sM}\text{L}(G_0)\,\overline{\otimes}\,\M\,\overline{\otimes}\,\mathbb M_n(\mathbb C)$, which contradicts  Lemma \ref{comultiply} (b). Similarly, we get that $\varepsilon_{i,2}( B)<B$ has finite index.

Let $1\leq i,j\leq k$. Since $p_i,p_j\in\Q$ are equivalent projections (as they are minimal projections of $1\,\overline{\otimes}\,\mathbb D_k(\mathbb C)\subset \Q=\sZ(\Q)\,\overline{\otimes}\,\mathbb M_k(\mathbb C)$) then  $p_j=zp_iz^*$, for some $z\in\sU(\Q)$. As $\Delta(u_{\widehat{h}})\omega_h\in\sU(p\sM p)$ normalizes $\Q$, we get that $z_h=\text{Ad}(\Delta(u_{\widehat{h}})\omega_h)(z)\in\sU(\Q)$.
Then $\Delta(u_{\widehat{h}})\omega_hp_j=\Delta(u_{\widehat{h}})\omega_hzp_iz^*=z_h\Delta(u_{\widehat{h}})\omega_hp_iz^*$. Using \eqref{conjug} we get that
\begin{align*}
\zeta_{j,h}u_{(\widehat{\varepsilon_{j,1}(h)},\widehat{\varepsilon_{j,2}(h)})}\otimes e_j&=u_j\Delta(u_{\widehat{h}})\omega_hp_ju_j^* \\ &=u_j(z_h\Delta(u_{\widehat{h}})\omega_hp_iz^*)u_j^*\\&=u_j(z_h(u_i^*(\zeta_{i,h}u_{(\widehat{\varepsilon_{i,1}(h)},\widehat{\varepsilon_{i,2}(h)})}\otimes e_i)u_i)z^*)u_j^*, \text{for every $h\in B$}.\end{align*}


{For $h\in B$, denote $\widetilde\zeta_h=u_j(z_h(u_i^*(\zeta_{i,h}u_{(\widehat{\varepsilon_{i,1}(h)},\widehat{\varepsilon_{i,2}(h)})}\otimes e_i)u_i)z^*)u_j^*$.
For a subset $F\subset B\times B$, denote by $P_F$ the orthogonal projection from $\text{L}^2(\sM)$ onto the $\|\cdot\|_2$-closure of the linear span of $\{(xu_g)\otimes y)\mid x\in\P\,\overline{\otimes}\,\P, g\in (\pi\times\pi)^{-1}(F), y\in\mathbb M_n(\mathbb C)\}$.
Since $z_h\in\sU(\Q)\subset\sum_{i=1}^k(\sP)_1x_i$, $\sP\subset \P\,\overline{\otimes}\,\P\,\overline{\otimes}\,\mathbb M_n(\mathbb C)$ and $\zeta_{i,h}\in \sU(\P\,\overline{\otimes}\,\P)$, by using that $\P\,\overline{\otimes}\,\P\subset\M\,\overline{\otimes}\,\M$ is a Cartan subalgebra and approximating $u_j,u_i,z$ in $\|\cdot\|_2$, we find a finite set $F\subset B\times B$ such that $P_{F\varepsilon_i(h)F}(\widetilde\zeta_h)\not=0$, for every $h\in B$. Since $\zeta_{j,h}\in \sU(\P\,\overline{\otimes}\,\P)$, the last displayed equation implies that $\varepsilon_j(h)\in F\varepsilon_i(h)F$, for every $h\in B$.}
If $g\in B\setminus\{e\}$, then as $ B$ is ICC and $\varepsilon_{i,1}( B),\varepsilon_{i,2}( B)$ have finite index in $ B$, the sets $\{\varepsilon_{i,1}(h)g\varepsilon_{i,1}(h)^{-1}\mid h\in B\}$ and $\{\varepsilon_{i,2}(h)g\varepsilon_{i,2}(h)^{-1}\mid h\in B\}$ are infinite.  Thus, the set $\{\varepsilon_i(h)g\varepsilon_i(h)^{-1}\mid h\in B\}$ is infinite, for every $g\in ( B\times B)\setminus\{(e,e)\}$. By  \cite[Lemma 7.1]{BV13} we derive the existence of $g\in B\times B$ such that $\varepsilon_j(h)=g\varepsilon_i(h)g^{-1}$, for every $h\in B$. This finishes the proof of Step 4.
\end{proof}

Step 4 thus gives a homomorphism $\delta=(\delta_1,\delta_2)\colon  B\rightarrow B\times B$ such that for every $1\leq i\leq k$, there is $g_i\in B\times B$ satisfying $\varepsilon_i(h)=g_i\delta(h)g_i^{-1}$, for every $h\in B$.
After replacing $\theta_i$ with $\alpha(g_i^{-1})\circ\theta_i$, we may assume that $\varepsilon_i=\delta$, for any $1\leq i\leq k$. Hence, \eqref{conjug}
rewrites as \begin{equation}\label{conjug1}\text{$u_i\Delta(u_{\widehat{h}})\omega_hp_iu_i^*=\zeta_{i,h}u_{(\widehat{\delta_1(h)},\widehat{\delta_2(h)})}\otimes e_i$, for every $1\leq i\leq k$ and $h\in B$.}\end{equation}
Let $u=\sum_{i=1}^ku_ip_i$ and $e=\sum_{i=1}^ke_i$. Then $u$ is a partial isometry with $uu^*=1\otimes 1\otimes e$, $u^*u=p$, $u\sP pu^*=\sP(1\otimes 1\otimes e)$. If $\zeta_h=\sum_{i=1}^k\zeta_{i,h}\otimes e_i\in\sU(\sP(1\otimes 1\otimes e))$, then \eqref{conjug1} gives

\begin{equation}\label{conjug2}
\text{$u\Delta(u_{\widehat{h}})\omega_hu^*=\zeta_h(u_{(\widehat{\delta_1(h)},\widehat{\delta_2(h)})}\otimes e)$, for every $h\in G$.}
\end{equation}

In particular, $t=(\tau\otimes\tau\otimes\text{Tr})(p)=\text{Tr}(e)=k$, and so $n=t=k$, $e=1$ and $p=1\otimes 1\otimes 1$. Thus, after replacing $\Delta$ with $\text{Ad}(u)\circ\Delta$, we have that $\zeta_h\in\sP$ and \eqref{conjug2} rewrites as

\begin{equation}\label{conjug3}
\text{$\Delta(u_{\widehat{h}})\omega_h=\zeta_h(u_{(\widehat{\delta_1(h)},\widehat{\delta_2(h)})}\otimes 1)$, for every $h\in B$.}
\end{equation}

{\it Step 5.} $\Q\subset \P\,\overline{\otimes}\,\P\,\overline{\otimes}\,\mathbb M_n(\mathbb C).$

\begin{proof}
Since $(\Delta(u_{\widehat{h}})\omega_h)_{h\in  B}\subset\sU(\sM)$ normalizes $\Q$ and $(\zeta_h)_{h\in B}\subset\sU(\sP)$,  \eqref{conjug3} implies that $(u_{(\widehat{\delta_1(h)},\widehat{\delta_2(h)})}\otimes 1)_{h\in B}$ normalizes $\Q$. Since $\sP\subset\Q$ and $(\Q)_1\subset\sum_{i=1}^k(\sP)_1x_i$, to prove the claim
 it suffices to argue that  for all $x,y\in \M\,\overline{\otimes}\,\M$ and $z\in (\M\,\overline{\otimes}\,\M)\ominus (\P\,\overline{\otimes}\,\P)$ we have
$$
\text{$\|E_{\P\,\overline{\otimes}\,\P}(xu_{(\widehat{\delta_1(h_m)},\widehat{\delta_2(h_m)})}zu_{(\widehat{\delta_1(h_m)},\widehat{\delta_2(h_m)})}^*y)\|_2\rightarrow 0$.}
$$

To prove this, we may assume that $x=u_a,y=u_b,z=u_g$, for $a,b,g\in G\times G$ with $g\notin (A^{( B)}\times A^{( B)})$. Write $(\pi\times\pi)(a)=(a_1,a_2),(\pi\times\pi)(b)=(b_1,b_2), (\pi\times\pi)(g)=(g_1,g_2)$. Since $(g_1,g_2)\not=(e,e)$, we have that $g_1\not=e$ or $g_2\not=e$.

Suppose that $g_1\not=e$.
For $h\in B$, denote $s_h=\|E_{\P\,\overline{\otimes}\,\P}(xu_{(\widehat{\delta_1(h)},\widehat{\delta_2(h)})}zu_{(\widehat{\delta_1(h)},\widehat{\delta_2(h)})}^*y)\|_2$. If $s_h=0$, for every $h\in B$, the assertion follows. Otherwise, if $s_h\not=0$ for $h\in B$,  then
$a_1\delta_1(h)g_1\delta_1(h)^{-1}b_1=e$.
 {This is because  $E_{\P\,\overline{\otimes}\,\P}(u_{(k_1,k_2)})\not=0$, for some $(k_1,k_2)\in G\times G$,  if and only if $(k_1,k_2)\in A^{(B)}\times A^{(B)}$ and if and only if $\pi(k_1)=\pi(k_2)=e$.}
In particular, there is $l\in  B$ such that $a_1lg_1l^{-1}b_1=e$. Moreover, if $s_{h_m}\not=0$ for some $m\in\mathbb N$, then  {$\delta_1(h_m)g_1\delta_1(h_m)^{-1}=a_1^{-1}b_1^{-1}=lg_1l^{-1}$ and therefore} $\delta_1(h_m)\in lB_0$, where $B_0=\text{C}_B(g_1)$. Let $G_0=\pi^{-1}(B_0)$.
In combination with \eqref{conjug3}, we get that $\Delta(u_{\widehat{h_m}})\omega_{h_m}\in (u_{\widehat{l}}\otimes 1\otimes 1)(\text{L}(G_0)\,\overline{\otimes}\,\M\,\overline{\otimes}\, \mathbb M_n(\mathbb C))$, for any such $m\in\mathbb N$.
Since $g_1\not=e$ and $B$ is ICC,  $B_0<B$ and thus $G_0<G$ has infinite index. Thus,  \eqref{v_h_m} implies that $\{m\in\mathbb N\mid s_{h_m}\not=0\}$ is finite, proving that $s_{h_m}\rightarrow 0$. Similarly, assuming that $g_2\not=e$ also implies that $s_{h_m}\rightarrow 0$. \end{proof}

Next, Step 5 implies that $\omega_h\in \P\,\overline{\otimes}\,\P\,\overline{\otimes}\,\mathbb M_n(\mathbb C)$, for every $h\in B$. Thus, if we denote $\eta_h=\zeta_h\text{Ad}(u_{(\widehat{\delta_1(h)},\widehat{\delta_2(h)})}\otimes 1)(\omega_h^*)$, then $\eta_h\in \P\,\overline{\otimes}\,\P\,\overline{\otimes}\,\mathbb  M_n(\mathbb C)$ and
\begin{equation}\label{Delt}
\text{$\Delta(u_{\widehat{h}})=\eta_h(u_{\widehat{\delta_1(h)}}\otimes u_{\widehat{\delta_2(h)}}\otimes 1)$, for every $h\in B$.}
\end{equation}

{\it Step 6.}
We may assume that $\delta_1=\delta_2=\text{Id}_{ B}$.

\begin{proof}
The proof is an adaptation of Step 5 in the proof of \cite[Theorem 8.2]{IPV10}. We first argue that we may assume that $\delta_1=\delta_2$. Using \eqref{Delta} and \eqref{Delt}, for every $h\in B$ we get that \begin{equation}\label{Del}\text{$\Delta_0(u_{\widehat{h}}\otimes 1)=U\psi(\Delta(u_{\widehat{h}})\otimes 1)U^*=U\psi(\eta_h\otimes 1)(u_{\widehat{\delta_1(h)}}\otimes 1\otimes u_{\widehat{\delta_2(h)}}\otimes 1)U^*$.}\end{equation}

Since ${\nabla}\circ\Delta_0=\Delta_0$, denoting $V=U^*{\nabla}(U)$  and $V_h=(\psi(\eta_h\otimes 1)^*V{\nabla}(\psi(\eta_h\otimes 1))$, we have

\begin{equation}\label{V_h}\text{$V_h(u_{\widehat{\delta_2(h)}}\otimes 1\otimes u_{\widehat{\delta_1(h)}}\otimes 1)=(u_{\widehat{\delta_1(h)}}\otimes 1\otimes u_{\widehat{\delta_2(h)}}\otimes 1)V$, for every $h\in B$.}\end{equation}

For $F_1,F_2\subset B$ finite, let $\mathcal H_{F_1,F_2}$ be the $\|\cdot\|_2$-closed linear span of $$\{u_{g_1}\otimes x_1\otimes u_{g_2}\otimes x_2\mid g_1\in\pi^{-1}(F_1),g_2\in\pi^{-1}(F_2),x_1,x_2\in\mathbb M_n(\mathbb C)\}$$ and $P_{F_1,F_2}$ be the orthogonal projection from $\text{L}^2(\M\,\overline{\otimes}\,\mathbb M_n(\mathbb C)\,\overline{\otimes}\,\M\,\overline{\otimes}\,\mathbb M_n(\mathbb C))$ onto $\mathcal H_{F_1,F_2}$. Let $F_1,F_2\subset\Gamma$ be finite sets such that $\|V-P_{F_1,F_2}(V)\|_2<1/2$. Since $\eta_h\in \P\,\overline{\otimes}\,\P\,\overline{\otimes}\,\mathbb M_n(\mathbb C)$ we get that $\psi(\eta_h\otimes 1),\zeta(\psi(\eta_h\otimes 1))\in \P\,\overline{\otimes}\,\mathbb M_n(\mathbb C)\,\overline{\otimes}\,\P\,\overline{\otimes}\,\mathbb M_n(\mathbb C)$, for every $h\in B$.  Since $\mathcal H_{F_1,F_2}$ is a $\P\,\overline{\otimes}\,\mathbb M_n(\mathbb C)\,\overline{\otimes}\,\P\,\overline{\otimes}\,\mathbb M_n(\mathbb C)$-bimodule, we further derive that $\|V_h-P_{F_1,F_2}(V_h)\|_2<1/2$, for every $h\in B$. In combination with \eqref{V_h}, for every $h\in B$ we get that $$\langle P_{F_1,F_2}(V_h)(u_{\widehat{\delta_2(h)}}\otimes 1\otimes u_{\widehat{\delta_1(h)}}\otimes 1),(u_{\widehat{\delta_1(h)}}\otimes 1\otimes u_{\widehat{\delta_2(h)}}\otimes 1)P_{F_1,F_2}(V)\rangle>0.$$

Note that $(u_{\widehat{g_1}}\otimes 1\otimes u_{\widehat{g_2}}\otimes 1)\mathcal H_{F_1,F_2}(u_{\widehat{h_1}}\otimes 1\otimes u_{\widehat{h_2}}\otimes 1)=\mathcal H_{g_1F_1h_1,g_2F_2h_2}$, for every $g_1,g_2,h_1,h_2\in B$.
Moreover, if $F_1\cap G_1=\emptyset$, then $P_{F_1,F_2}P_{G_1,G_2}=0$. Thus, we get that
$F_1\delta_2(h)\cap \delta_1(h)F_1\not=\emptyset$, for every $h\in B$.
Since $ B$ is ICC and $\delta_1( B)< B$ has finite index, it follows that there exists $g\in B$ such that $\delta_2(h)=g\delta_1(h)g^{-1}$, for every $h\in B$ (see \cite[Lemma 7.1]{BV13}). Thus, after replacing $\Delta$ by $\text{Ad}(1\otimes u_g^*\otimes 1)\circ\Delta$ and $\eta_h$ by $\text{Ad}(1\otimes u_g^*\otimes 1)(\eta_h)\in \P\,\overline{\otimes}\,\P\,\overline{\otimes}\,\mathbb M_n(\mathbb C)$, we may assume that $\delta_1=\delta_2$.
Put $\delta=\delta_1=\delta_2$.

 To argue that $\delta$ is inner, define $X_1, X_1^h, X_2, X_2^h\in  \,\overline{\otimes}\,_{k=1}^3(\M\,\overline{\otimes}\,\mathbb M_n(\mathbb C))$, for $h\in B$, as follows:
 $X_1=(U\otimes 1\otimes 1)^*(\Delta_0\otimes\text{Id})(U)^*$,
$X_1^h=(\Delta_0\otimes\text{Id})(U\psi(\eta_h\otimes 1))(U\psi(\eta_{\delta(h)}\otimes 1))\otimes 1\otimes 1)$, $X_2=(1\otimes 1\otimes U)^*(\text{Id}\otimes\Delta_0)(U)^*$, and $X_2^h=(\text{Id}\otimes\Delta_0)(U\psi(\eta_h\otimes 1))(1\otimes 1\otimes U\psi(\eta_{\delta(h)}\otimes 1))$.

Then for every $h\in B$ we have that
$$\text{$(\Delta_0\otimes\text{Id})\Delta_0(u_{\widehat{h}}\otimes 1)=X_h^1(u_{\widehat{\delta(\delta(h))}}\otimes 1\otimes u_{\widehat{\delta(\delta(h))}}\otimes 1\otimes u_{\widehat{\delta(h)}}\otimes 1)X_1$ and}$$
$$\text{$(\text{Id}\otimes\Delta_0)\Delta_0(u_{\widehat{h}}\otimes 1)=X_h^2(u_{\widehat{\delta(h)}}\otimes 1\otimes u_{\widehat{\delta(\delta(h))}}\otimes 1\otimes u_{\widehat{\delta(\delta(h))}}\otimes 1)X_2$.} $$

Since $(\Delta_0\otimes\text{Id})\circ\Delta_0=(\text{Id}\otimes\Delta_0)\circ\Delta_0$, by adapting the above argument we can find $F\subset B$ finite such that $F\delta(\delta(h))\cap \delta(h)F\not=\emptyset$, for every $g\in B$. Since $\delta( B)< B$ has finite index and $ B$ is ICC, this implies that there is $g\in B$ such that $\delta(h)=ghg^{-1}$, for every $h\in B$ (see \cite[Lemma 7.1]{BV13}).  Thus, after replacing $\Delta$ by $\text{Ad}(u_g^*\otimes u_g^*\otimes 1)\circ\Delta$ and $\eta_h$ by $\text{Ad}(u_g^*\otimes u_g^*\otimes 1)(\eta_h)\in \P\,\overline{\otimes}\,\P\,\overline{\otimes}\,\mathbb M_n(\mathbb C)$, we may assume that $\delta_1=\delta_2=\text{Id}_{ B}$, that is \begin{equation}\text{$\Delta(u_{\widehat{h}})=\eta_h(u_{\widehat{h}}\otimes u_{\widehat{h}}\otimes 1)$, for every $h\in B$.}
\end{equation}
 This finishes the proof of Step 6.
\end{proof}

To finish the proof of Theorem \ref{superr},
let $g\in G$. Let $h=\pi(g)\in B$ and $a=g\widehat{h}^{-1}\in A^{(I)}$.
Then $\Delta(u_g)=\Delta(u_a)\Delta(u_{\widehat{h}})=\Delta(u_a)\eta_h(u_{\widehat{h}}\otimes u_{\widehat{h}}\otimes 1)=\Delta(u_a)\eta_h(u_a\otimes u_a\otimes 1)^*(u_g\otimes u_g\otimes 1)$. Thus, if we denote $w_g=\Delta(u_a)\eta_h(u_a\otimes u_a\otimes 1)^*$, then $w_g\in \sU(\P\,\overline{\otimes}\,\P\,\overline{\otimes}\,\mathbb M_n(\mathbb C))$ and \begin{equation}\label{w_g'}
\text{$\Delta(u_g)=w_g(u_g\otimes u_g\otimes 1)$, for every $g\in G$.}
\end{equation}

Consider the action $G\curvearrowright^{\gamma} \P\,\overline{\otimes}\,\P$ given by $\gamma_g=\text{Ad}(u_g\otimes u_g)$, for $g\in G$. Lemma \ref{coind} implies that $\gamma$ is  {conjugate} to an action $G\curvearrowright (Y^J,\nu^J)$ built over  $G\curvearrowright J=I\times\{1,2\}$ given by $g\cdot (i,j)=(\pi(g)\cdot i,j)$, for every $g\in G$ and $(i,j)\in J$. Since the action $G\curvearrowright J$ has infinite orbits, $\gamma$ is weakly mixing. Moreover,
 \eqref{w_g'} gives that $w_{gh}=w_g(\gamma_g\otimes\text{Id})(w_h)$, for every $g,h\in G$. Therefore, $(w_g)_{g\in G}$ is a $1$-cocycle for $\gamma\otimes\text{Id}$, with $\text{Id}$ the trivial action on ${\mathbb M_n(\mathbb C)}$. 

Since  $G$ has property (T), Theorem \ref{builtover} gives $u\in \sU(\P\,\overline{\otimes}\,\P\,\overline{\otimes}\,\mathbb M_n(\mathbb C))$ and a homomorphism $\xi\colon G\rightarrow\sU_n(\mathbb C)$ such that $w_g=u^*(1\otimes 1\otimes\xi_g)(\gamma_g\otimes\text{Id})(u)$, for every $g\in G$. Thus, after replacing $\Delta$ by $\text{Ad}(u)\circ\Delta$, \eqref{w_g'} rewrites as   \begin{equation}\label{xi_g}\text{$\Delta(u_g)=u_g\otimes u_g\otimes\xi_g$, for every $g\in G$.}\end{equation}

Let $\sN$ be the von Neumann algebra generated by $\{u_g\otimes u_g\otimes x\mid g\in G,x\in\mathbb M_n(\mathbb C)\}$.
Then $\Delta(\M)\subset\sN\subset\Delta(\M)\mathbb M_n(\mathbb C)$. By Lemma \ref{comultiply} (c), $\sN=\Delta(\M)$, so $1\,\overline{\otimes}\,1\,\overline{\otimes}\,\mathbb M_n(\mathbb C)\subset\Delta(\M)$. In combination with \eqref{xi_g}, this implies that $n=1$ and hence $t=1$.
Also, $\xi_g\in\mathbb T$ and
\begin{equation}\label{untwist}
\text{$\Delta(u_g)=\xi_g(u_g\otimes u_g)$, for every $g\in G$.}
\end{equation}
Moreover, there is $\Omega\in \sU(\M\,\overline{\otimes}\,\M)$ so that $\Delta_0=\text{Ad}(\Omega)\circ\Delta$.
By \eqref{untwist}, $\Omega(u_g\otimes u_g)\Omega^*\in\Delta_0(\M)$, for every $g\in G$. Since $G$ is ICC, the unitary representation $(\text{Ad}(u_g))_{g\in G}$ of $G$ on $\text{L}^2(\M)\ominus\mathbb C1$ is weakly mixing. By applying \cite[Lemma 3.4]{IPV10} we conclude that there are $w\in\sU(\M)$ and an isomorphism $\rho\colon G\rightarrow H$ such that $u_g=\xi_gwv_{\rho(g)}w^*$, for every $g\in G$.
\end{proof}

\end{document}